\documentclass{amsart}
\usepackage{amssymb,latexsym,graphicx}

\newcommand{\R}{{\mathbb{R}}}

\newcommand{\Z}{{\mathbb{Z}}}

\DeclareMathOperator{\Arf}{Arf}  

\newcommand{\cW}{\mathcal{W}}

\newcommand{\cI}{\mathcal I}
\newcommand{\dash}{{\mbox{--}}}

\newcommand{\tree}[3]{\text{\Large {$ {\text{\normalsize $ #1$}}-\!\!\!<^{#2}_{#3}$}}}

\newtheorem{proposition}{Proposition}[section]
\newtheorem{theorem}[proposition]{Theorem}
\newtheorem*{theorem*}{Theorem}
\newtheorem{lemma}[proposition]{Lemma}

\newtheorem{corollary}[proposition]{Corollary}

\theoremstyle{definition}
\newtheorem{example}[proposition]{Example}
\newtheorem{definition}[proposition]{Definition}

\theoremstyle{remark}
\newtheorem{remark}[proposition]{Remark}

\newcommand{\iinfty}{{\mathchoice
{\begin{minipage}{.15in}\includegraphics[width=.15in]{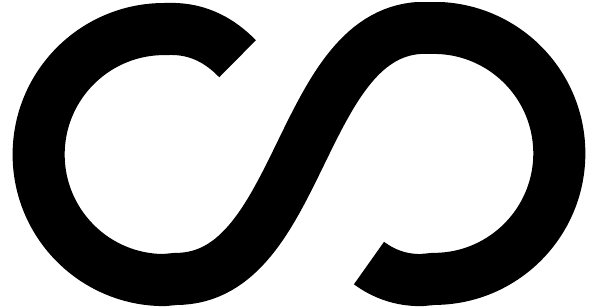}\end{minipage}}
{\begin{minipage}{.13in}\includegraphics[width=.13in]{infty2.pdf}\end{minipage}}
{\begin{minipage}{.11in}\includegraphics[width=.11in]{infty2.pdf}\end{minipage}}
{\begin{minipage}{.08in}\includegraphics[width=.08in]{infty2.pdf}\end{minipage}}
}}

\begin{document}

\markboth{Conant, Schneiderman and Teichner}
{Cochran invariants and Whitney Towers}


\title{COCHRAN'S $\beta^i$-INVARIANTS VIA TWISTED WHITNEY TOWERS}

\author{JIM CONANT, ROB SCHNEIDERMAN and PETER TEICHNER}

\address{Dept. of Mathematics, University of Tennessee, Knoxville;\\ Dept. of Mathematics and Computer Science, Lehman College, City University of New York;\\ Max-Planck-Institut f\"ur Mathematik, Bonn, and University of California, Berkeley.}

\maketitle

\begin{abstract}
We show that Tim Cochran's invariants $\beta^i(L)$ of a $2$-component link $L$ in the $3$--sphere can be computed as intersection invariants of certain 2-complexes in the $4$--ball with boundary $L$. These 2-complexes are special types of twisted Whitney towers, which we call {\em Cochran towers}, and which exhibit a new phenomenon: A Cochran tower of order $2k$ allows the computation of the $\beta^i$ invariants for all $i\leq k$, i.e.\ simultaneous extraction of invariants from a Whitney tower at multiple orders. This is in contrast with the order $n$ Milnor invariants  (requiring order $n$ Whitney towers) and consistent with Cochran's result that the $\beta^i(L)$ are integer lifts of certain Milnor invariants. 
\end{abstract}



\section{Introduction and statement of results}
In 1954, John Milnor defined his $\mu$-invariants of a link $L=(L_1,\dots, L_m)$ in 3-space  \cite{Mi} by looking inductively at the terms in the lower central series of the link group $\pi_1(\R^3 \smallsetminus L)$, and comparing with the link group of the unlink. For example, the {\em order} $0$ Milnor invariants are just the linking numbers $\mu_{ij}(L)$ between components $L_i$ and $L_j$ of the link $L$. Moreover, Milnor showed that $\mu_{123}$ detects the Borromean rings, a {\em Bing double} of the Hopf link, and that his higher-order invariants detect iterated Bing doublings of the Hopf link. 

The Milnor invariants $\mu_\cI(L)$ of {\em order} $n\geq 0$ are labeled by a multi-index $\cI=\{i_1 i_2\dots i_{n+2}\}$ with $i_k\in\{1,\dots,m\}$. They are integers, well-defined only modulo the gcd of the $\mu_{\cI'}(L)$ for all proper subindices $\cI'\subset \cI$. For example, if $L$ is a two component link with linking number $\mu_{12}(L)=1$, then all the higher-order Milnor invariants of $L$ are completely ill-defined. Nevertheless, Milnor invariants turn out to be a central tool in 3- and 4-dimensional topology, particularly because of their concordance invariance. 

For example, if the topological surgery sequence is exact in dimension~4 for free fundamental groups, then the {\em Whitehead double} of any link $L$ with trivial linking numbers is topologically slice. However, this last statement is currently only known to hold for links $L$ with $\mu_\cI(L)=0$ for any multi-index $\cI=\{i_1 i_2\dots i_{n+2}\}$ in which at most one index appears more than once (and at most twice) in $\cI$ \cite{FT}. 

In 1985, Tim Cochran discovered a beautiful method of lifting certain Milnor invariants to well-defined integers  \cite{Co}: Given a $2$-component link $L=(L_1,L_2)$ with $\mu_{12}(L)=0$, he  first defined its {\em derived link} $D(L)$ by forming a knot as the intersection  
of Seifert surfaces for the components (each in the complement of the other component), and then taking this knot in place of $L_2$ to yield the new $2$-component link $D(L)$. 

Tim then defined  $\beta^1(L)\in\Z$ as the {\em Sato-Levine invariant} of $L$ \cite{Sa}, which is the twisting of the intersection knot, inherited from either Seifert surface. What are today known as {\em higher-order Cochran invariants} were defined recursively via the formula $\beta^i(L)=\beta^{i-1}(D(L))\in\Z$. Amazingly, these integers are well-defined for $i\geq 1$, and Tim showed \cite[Thm.6.10]{Co1} that they are lifts of the following Milnor invariants:
 \[
  \beta^i(L) \equiv \mu_{1^{2i}2^2}(L):=\mu_{1\dots_{2i}\dots 122}(L) \mod \gcd\{\mu_{1^{2k}2^2}(L), k<i\}
\]
Only $\beta^1(L)=\mu_{1122}(L)$ was known to be a well-defined integer for linking number zero links.

In this note, we re-interpret Cochran's invariants $\beta^i(L)$ in terms of intersection invariants of certain 2-complexes in the 4--ball $B^4$, with boundary $L\subset S^3$. These 2-complexes are special kinds of {\em twisted Whitney towers} \cite{CST,CST2} and we propose to call them {\em Cochran towers}. 

The notions necessary to understand our theorem below will be given in the remainder of the introduction but first we would like to state the main result: 

\begin{theorem} \label{thm:main}
For any $k\geq 1$, a link $L=(L_1,L_2)$ in $S^3$ bounds a Cochran tower of order $2k$ in $B^4$ if and only if $L$ has trivial linking number and the Arf invariant of $L_1$ vanishes. Moreover, given a Cochran tower $C$ of order $2k$ with boundary $L$ then for any $i\leq k$, the Cochran invariants can be computed as follows:
\[
\beta^i(L) = \sum \omega(W_J),
\]
where  the sum is over all Whitney disks $W_J$ in $C$ with $J^\iinfty \cong t^\iinfty_i$ and the \emph{twisting} 
$\omega(W_J)\in\Z$ is the relative Euler number of $W_J$. 
\end{theorem}
Here $J^\iinfty$ is a trivalent tree associated to the Whitney disk $W_J$ in $C$ and the relevant tree $t^\iinfty_i$ is shown on the right  hand side of Figure~\ref{fig:t-tree}.

\begin{remark} If the ($\Z/2$-valued) Arf invariant $\Arf(L_1)\neq 0$, one can change $L$ by
 tying a small trefoil knot into $L_1$. This does not alter $\beta^1$, nor the derived link $D(L)$ (and hence all $\beta^i(L)$ are unchanged) but allows one to build a Cochran tower on this new link to compute all $\beta^i(L)$.
\end{remark}

\begin{figure}[h]
\includegraphics[width=\linewidth]{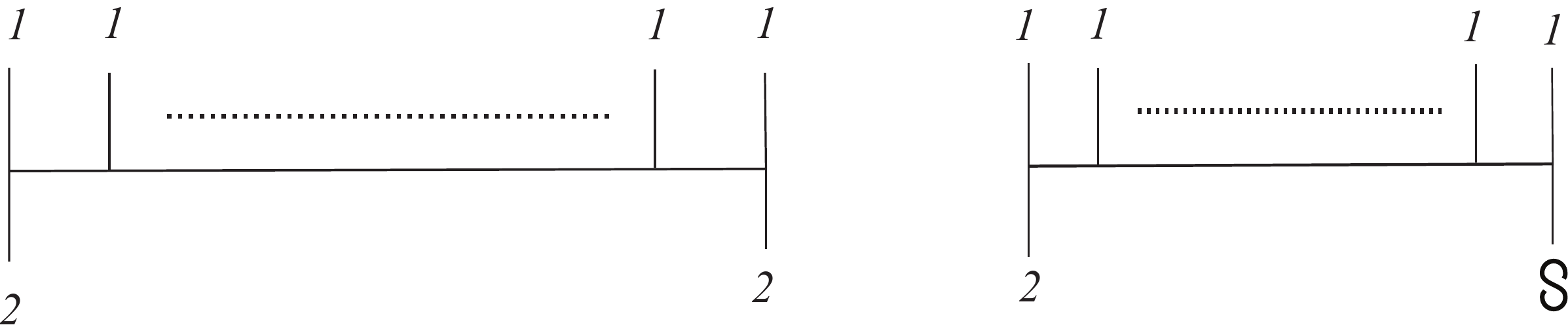}
\caption{The tree $t_i$ on the left has two vertices labeled $2$ and $i$ vertices labeled $1$. The tree $t_i^{\infty}$ on the right has $i$ vertices labeled $1$, one vertex labeled $2$, and one vertex labeled by the {\em twist} symbol.} \label{fig:t-tree}
\end{figure}

A geometric interpretation similar to Theorem~\ref{thm:main} was given in our earlier work \cite{CST2} for the {\em first non-vanishing} Milnor invariants $\mu_\cI(L)\in\Z$ of a given link $L$. This means that $\mu_{\cI'}(L)=0$ for all proper subsets $\cI'\subset \cI$.
These interpretations use our theory of {\em twisted Whitney towers} in the 4--ball with boundary $L$, which is surveyed in \cite{CST0} and detailed in \cite{CST,CST2}. We will sketch next those aspects which are relevant to the current discussion, and point the reader to the appropriate references for more information. 
The smooth category is assumed throughout, except when otherwise specified.

\subsection{Twisted Whitney towers and their intersection forests}\label{subsec:int-forest}
Roughly speaking, a twisted Whitney tower (with boundary a link $L\subset S^3$) is a finite 2-complex $\cW\subset B^4$ with $\cW\cap S^3 = L$, which is the union of a finite number of transverse disks with carefully chosen boundary conditions. More precisely, $\cW$ is formed by taking generic disks $W_j$ in the $4$--ball with $\partial W_j=L_j$ and then adding Whitney disks for pairs of intersection or self-intersection points among the $W_j$ (if possible). One continues to inductively add {\em higher-order} Whitney disks for pairs of (self)-intersection points among previously included disks, to arrive at a Whitney tower after finitely many steps. 

The main use of a Whitney tower $\cW$ comes from its {\em intersection forest} $t(\cW)$ \cite[Sec.2.5]{CST2} which is a disjoint
union of trivalent labeled trees, as we review next. As mentioned after Definition~\ref{def:t(W)}, $t(\cW)$ represents obstructions to successfully carrying out Whitney moves that would lead to slice disks for the link $L$. 

\begin{figure}[h]
\begin{center}
\includegraphics[width=.95\linewidth]{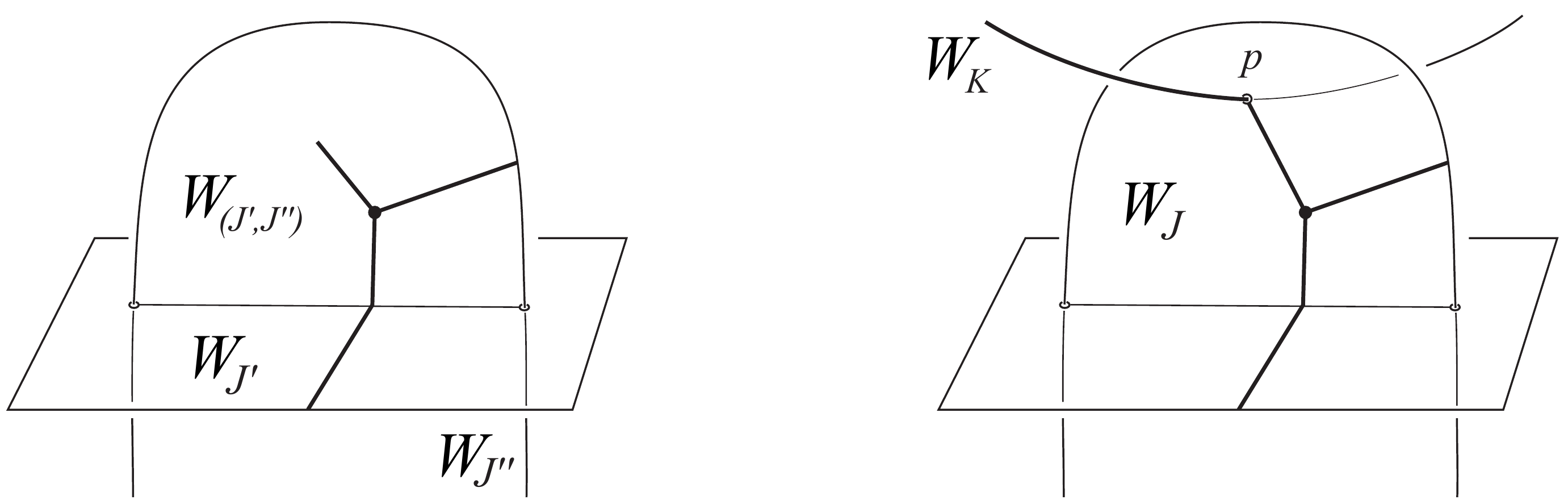}
\end{center}
\caption{Left: A Whitney disk $W_{(J',J'')}$ pairing intersections between $W_{J'}$ and $W_{J''}$ with (part of) its associated rooted tree $J=(J',J'')$.
Right: Part of the (unrooted) tree associated to an unpaired intersection $p\in W_J\cap W_K$.}
\label{fig:w-tree}
\end{figure}

To define $t(\cW)$, start by associating a rooted unitrivalent tree $J$ to each Whitney disk $W\subset\cW$ as follows: Take a univalent {\em root} vertex of $J$ sitting in the interior of $W$ along with an edge to the adjacent trivalent vertex, and take the other edges of $J$ to be sheet-changing paths bifurcating down through the lower-order Whitney disks until arriving at the boundary components $L_j$. This yields $j$-labeled univalent vertices in $J$ (the root is the only unlabeled univalent vertex). 

The notation $W_J$ indicates that $J$ is the rooted tree associated to the Whitney disk $W$, and we identify rooted labeled trees with non-associative bracketings of the index set formed by the link-components. So $W_{(J',J'')}$ pairs intersections between
$W_{J'}$ and $W_{J''}$, and the rooted tree $(J',J'')$ is formed by identifying the roots of $J'$ and $J''$ to a single vertex and sprouting the rooted edge of $(J',J'')$ from this vertex (Figure~\ref{fig:w-tree}, left). The {\em order} of $W_J$ is defined to be the order of $J$, which is the number of trivialent vertices in $J$. In particular, for a singleton index $j$, the order of $W_j$ is zero and its rooted tree is the  edge $\,\dash\dash\dash \,j$. 

\begin{figure}[h]
\begin{center}\includegraphics[width=.95\linewidth]{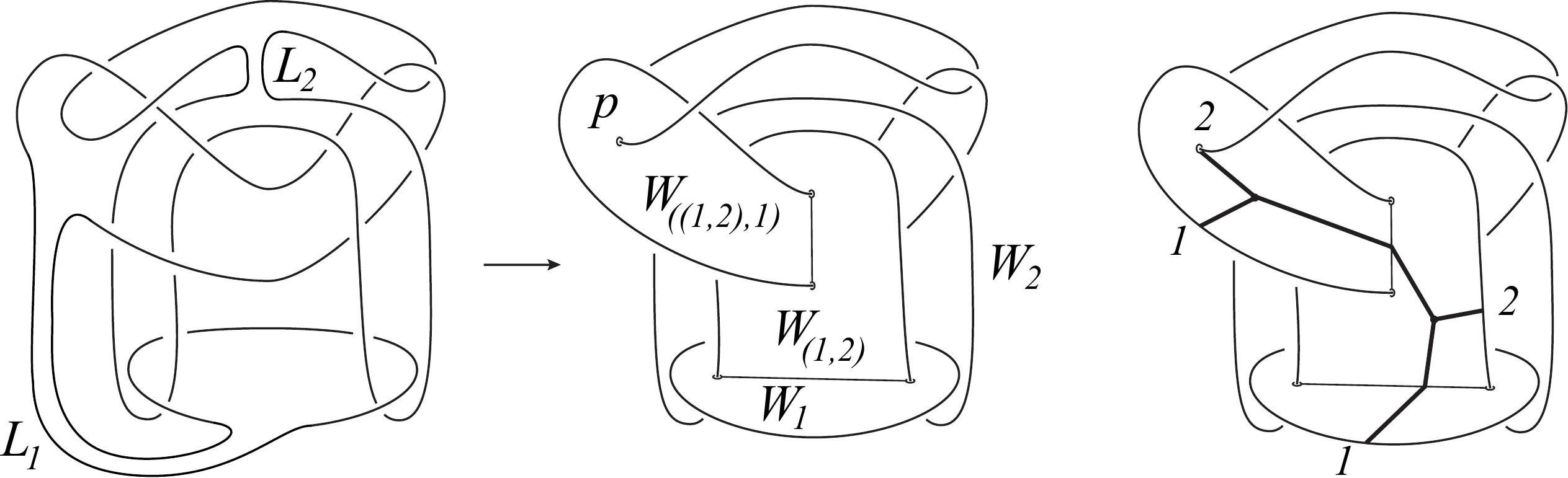}
\end{center}
\caption{Left: An internal band sum $L=(L_1,L_2)\subset S^3$ of a Bing-double of the Hopf link.  Center: Moving into $B^4$, the bands have `dissolved' into $1$-handles, leaving two $0$-handles for each disk $W_j$ bounding $L_j$, with one $0$-handle for $W_1$ visible as a horizontal disk.
The Whitney disk $W_{(1,2)}$ pairs $W_1\cap W_2$, and an there is a single unpaired intersection $p$ between $W_2$ and the Whitney disk $W_{((1,2),1)}$ pairing $W_{(1,2)}\cap W_1$. Right: The tree associated to $p$ is $t_2= \,^1_2>\!\!\!-\!\!\!\!-\!\!\!\!\!-\!\!\!<^{\,1}_{\,2} $ from Figure~\ref{fig:t-tree}. 
}
\label{fig:bing-hopf-example-with-tower-banded}
\end{figure}

Now each tree in $t(\cW)$ corresponds to one of two kinds of {\em problems} in $\cW$: 
A transverse intersection $p\in W_J\pitchfork W_K$ is a problem if $p$ is not paired by a higher-order Whitney disk. 
To such a $p$ is associated the labeled tree gotten by identifying the roots of $J$ and $K$ to a 
single (interior) point, 
(Figure~\ref{fig:w-tree}, right; and Figure~\ref{fig:bing-hopf-example-with-tower-banded}, right). 
The other kind of problem that can occur is that a Whitney disk $W_J$ in $\cW$ may be {\em twisted}, i.e. have a non-zero relative Euler number $\omega(W_J)\in\Z$ with respect to the standard framing of its boundary (e.g.~\cite[Sec.2.2]{CST}). This integer $\omega(W_J)$ is called the \emph{twisting} of $W_J$, and if $\omega(W_J)\neq 0$, then a Whitney move guided by $W_J$ will create new self-intersections, even if $W_J$ happens to be embedded.
This second problem contributes to $t(\cW)$ a tree $J^\iinfty$ defined by labeling the root vertex of $J$ with the symbol $\iinfty$ standing for ``twist''. 

A Whitney disk with twisting equal to zero is said to be \emph{framed}.
Since it can always be arranged (by \emph{splitting} $\cW$ if needed \cite[Sec.2.6]{CST2}) that 
unpaired intersections $p$ occur between framed disks, we frequently refer to the tree associated to such a $p$ as a \emph{framed} tree, to differentiate from the trees $J^\iinfty$ associated to twisted Whitney disks, which we usually call \emph{twisted trees}, or $\iinfty$-trees.

\begin{figure}[h]
\begin{center}
\includegraphics[width=.8\linewidth]{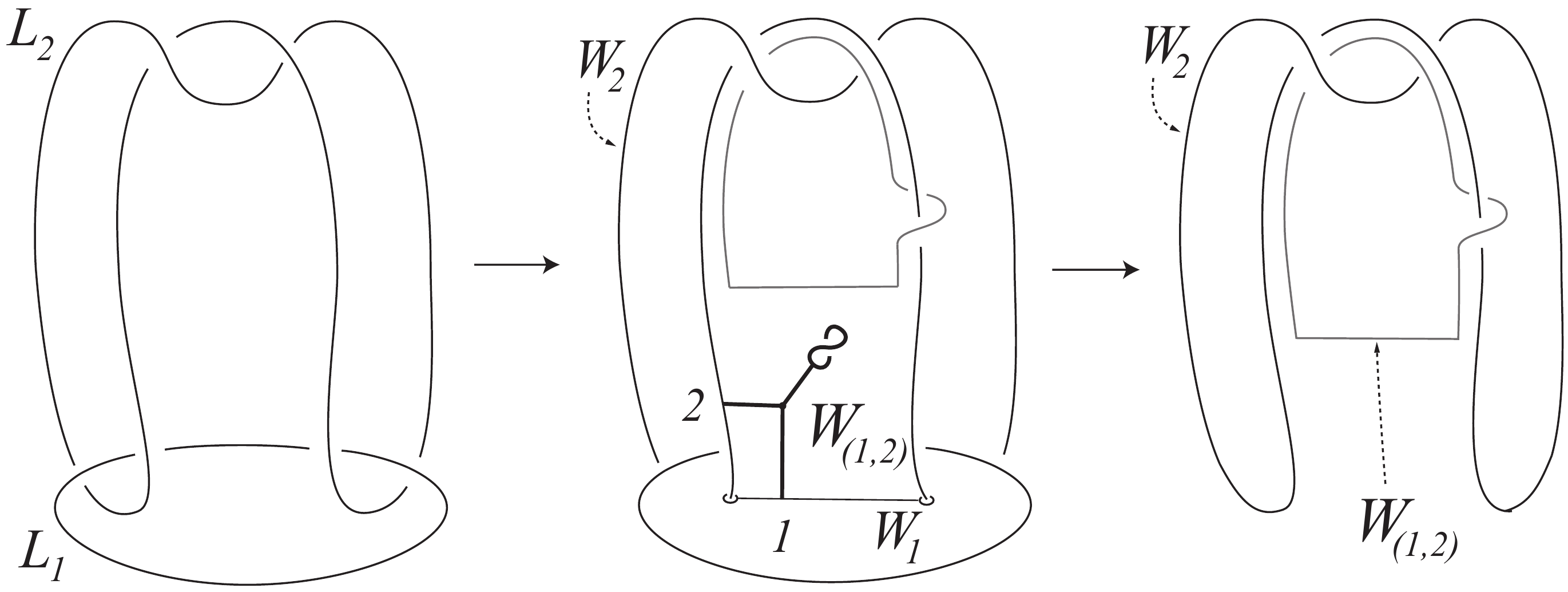}
\end{center}
\caption{Left: The Whitehead link $L=(L_1,L_2)\subset S^3$. Center: Moving into $B^4$, (most of) an embedded twisted Whitney disk $W_{(1,2)}$ pairing $W_1\cap W_2$, and the associated twisted tree. Right: Continuing into $B^4$, the rest of $W_{(1,2)}$ and $W_2$ are described by capping off this unlink with two embedded disks (not shown). The twisting $\omega(W_{(1,2)})=+1$ is evident in the right-handed twist 
(see e.g. [5, sec. 6]).
\label{fig:Whitehead-twisted-tree}}
\end{figure}

\begin{definition}\label{def:t(W)}
For a twisted Whitney tower $\cW$, define the {\em intersection forest} $t(\cW)$ as the disjoint union of (isomorphism classes of) such trivalent, labeled trees, one for each problem in $\cW$. 
\end{definition}

Note that if $\cW$ has no problems, i.e.\ $t(\cW)$ is empty, then one can do a sequence of (embedded, framed) Whitney moves to produce slice disks for the link $L$ on the boundary.

\begin{remark}\label{rem:t(W)-without-coefficients}
Comparing this definition with the definition of $t(\cW)$ given in \cite[Sec.2.5]{CST2}, the reader will notice that there the framed trees carry a vertex-orientation and a coefficient $\pm 1$, coming from orientations induced by the original link components (involving \emph{antisymmetry relations}), and the twisted trees carry the integer coefficients coming from the twisting of the corresponding Whitney disks (and independent of disk orientations). We do not need this extra data to define Cochran towers, although it is used in the obstruction theory discussed below.
\end{remark}

\begin{remark}\label{rem:not-using-twisted-terminology}
As per Definition~\ref{def:w-tower-order} below, twisted Whitney disks are allowed in \emph{twisted Whitney towers} (e.g.~\cite{CST,CST2}), while earlier papers defined \emph{Whitney towers} with the requirement that Whitney disks be \emph{framed}. For 
conversational ease we may sometimes let ``Whitney tower'' refer to either 
twisted or framed Whitney towers in a general discussion that applies to both settings, hopefully when no confusion will result.
\end{remark}

\subsection{Obstruction theory for twisted Whitney towers}\label{subsec:order-raising-intro}
 
\begin{definition}\label{def:w-tower-order}
Let $\cW$ be a twisted Whitney tower.
\begin{enumerate}
\item\label{def-item:tree-order}
The \emph{order} of any tree in $t(\cW)$ is the number of trivalent vertices. (This applies to both framed and twisted trees.)
\item\label{def-item:twisted-w-tower-order}
If all framed trees in $t(\cW)$ are of order $\geq n$ and all twisted trees in $t(\cW)$ are of order $\geq \frac{n}{2}$ then $\cW$ is a \emph{twisted Whitney tower of order $n$}.
\end{enumerate}

\end{definition}
This notion of order for twisted Whitney towers arose in our work on Milnor invariants \cite{CST2}, where we discovered that for any order $n$ twisted $\cW$ bounded by $L$ the order $n$ framed trees and order $\frac{n}{2}$ twisted trees (for $n$ even) in $t(\cW)$ contribute to $\mu_{i_1 i_2\dots i_{n+2}}(L)$ (with the twisted trees corresponding to certain multi-indices with order 2 symmetry).

More precisely, we showed in \cite{CST2} how the first non-vanishing $\mu$-invariants can be computed from $t(\cW)$ modulo certain relations, all of which can be realized via geometric maneuvers preserving the order of $\cW$ (without changing its boundary $L$).
Most prominently, the \emph{geometric IHX-relations}, or 4-dimensional Jacobi-identities, can be used to change $t(\cW)$ by replacing a tree containing an $I$-shaped subtree with two trees of the same order that only differ locally by $H$- and $X$-shaped subtrees, plus a number of trees of higher order \cite{CST}. 
For instance, at the cost of creating higher-order trees, the geometric IHX-relations can be used to modify an order $n$ twisted Whitney tower so that all framed trees in $t(\cW)$ with two $2$-labels and $n$ $1$-labels are isomorphic to $t_n$ in Figure~\ref{fig:t-tree}, and all twisted trees with one $2$-label and $\frac{n}{2}$ $1$-labels are isomorphic to $t^\iinfty_{\frac{n}{2}}$ if $n$ is even.

It follows from \cite{CST2} that non-trivial order $n$ Milnor invariants of $L$ are obstructions to $L$ bounding twisted Whitney towers of order greater than $n$, and it is not clear how order $k$ Milnor invariants for $k>n$ might be related to intersection forests of order $n$ Whitney towers. 
Since the specific Milnor invariants we are focusing on in this paper occur in multiple orders and can be lifted to the well-defined Cochran invariants $\beta^i(L)$, we arrive at the question of whether we can compute these in terms of twisted Whitney towers bounded by $L$. 
The answer, given by Theorem~\ref{thm:main}, is that indeed we can, and we simply add the twistings of certain Whitney disks. However, there is a catch:
For a Cochran tower $\cW$, the types of trees in $t(\cW)$ must be restricted in new ways!

\subsection{Cochran towers}\label{subsec:cochran-towers}
To define Cochran towers, we start with the following definitions for Whitney towers bounded by 2-component links, so all trees are labeled by $1$, $2$ or $\iinfty$.

\begin{definition}\label{def:bad} 
A trivalent labeled tree is called {\em $\beta$-bad} if at most one univalent vertex is not labeled by 1, or if it isomorphic to $t_i$ in Figure~\ref{fig:t-tree} for some $i$.   
\end{definition}

\begin{definition}\label{def:Cochran tower}
A twisted Whitney tower $\cW$  is a \emph{Cochran tower of order $n$} if all framed $\beta$-bad trees in $t(\cW)$ are of order $> n$ and all twisted $\beta$-bad trees in $t(\cW)$ are of order $> \frac{n}{2}$. Consequently, a twisted Whitney tower $\cW$  is a \emph{Cochran tower of infinite order} if $t(\cW)$  does not contain any $\beta$-bad trees at all.
\end{definition}

We observe that:
\begin{itemize}
\item  The framed $\beta$-bad trees are the framed trees having no $2$-labels, the framed trees having a single $2$-label, and the $t_i$.

\item  The twisted $\beta$-bad trees are the twisted trees having no $2$-labels.
\end{itemize}

Illustrations are given in Examples \ref{ex:infmany}, \ref{ex:no-2-labels-are-bad} 
and \ref{ex:t2-bad} showing how the presence of lower-order $\beta$-bad trees in a twisted Whitney tower can create indeterminacies in the computation of higher-order $\beta$-invariants.

The order zero trees $1\,\dash\dash\dash \,2$ in $t(\cW)$ give the linking number of $L$, while $\tree{1}{1}{1}$ and $\tree{1}{1}{\iinfty}$ are responsible for $\Arf(L_1)$. This explains the easy direction in the first sentence of Theorem~\ref{thm:main}. The hard direction is proven by applying all the tricks that we learned over the years, on how to raise the order of a Whitney tower. In particular, those trees that vanish modulo IHX can be exchanged for intersections of arbitrarily high order without changing the $\beta^i$.

\begin{remark}\label{rem:multiple-order}
Notice that, via Theorem~\ref{thm:main}, a Cochran tower of order $2k$ allows the computation of $\beta^i(L)$ for all $i\leq k$, i.e.\ simultaneous extraction of invariants from a Whitney tower at multiple orders, in contrast with the order $n$ Milnor invariants which require $L$ to bound an order $n$ twisted Whitney tower. To our knowledge this is the first example of this kind of computation. Moreover, in an infinite order Cochran tower, the $\beta^i$ can be computed in all orders, as in the next Example~\ref{ex:Whitehead}.
\end{remark}

\begin{example}\label{ex:Whitehead} (The Whitehead link) 
Figure~\ref{fig:Whitehead-twisted-tree} describes a twisted Whitney tower $\cW$ bounded by a (positive) Whitehead Link $L$, with $t(\cW)=\tree{1}{2}{\iinfty}=t_1^\iinfty$.
Since $t(\cW)$ contains no $\beta$-bad trees, $\cW$ is a Cochran tower of infinite order, and it follows from Theorem~\ref{thm:main} that $\beta^1(L)=1$, and  $\beta^i(L)=0$ for all $i>1$.
In fact, if one creates Seifert surfaces by adding tubes to the obvious disks bounded by the components  of $L$ (in order to make the Seifert surfaces disjoint from the other component), then the derived link $D(L)$ is the same unlink as in the right-hand side of Figure~\ref{fig:Whitehead-twisted-tree}. 

\end{example}

\begin{example} \label{ex:infmany} (Infinitely many nontrivial Cochran invariants, and why the linear framed trees with a single $2$-label at one end  (Figure~\ref{fig:iterating}, lower right) are $\beta$-bad.) 
Let $L$ be the link where the first component is a trefoil knot considered as the boundary of a genus one Seifert surface with a $+1$ twist in each band. The second component is a meridian to one of the two bands. See Figure~\ref{fig:infmany}, and \cite[Ex.4.6]{Co}. The derived link $D(L)$ is actually the same link but with the second component linking the other band, and this is isotopic to the original link. So $D(L)=L.$ Similarly to the Whitehead link, the Seifert surfaces for $L$ intersect in a knot with twisting $+1$, so $\beta^i(L)=1$ for all $i$.

\begin{figure}[h]
\begin{center}
\includegraphics[width=\linewidth]{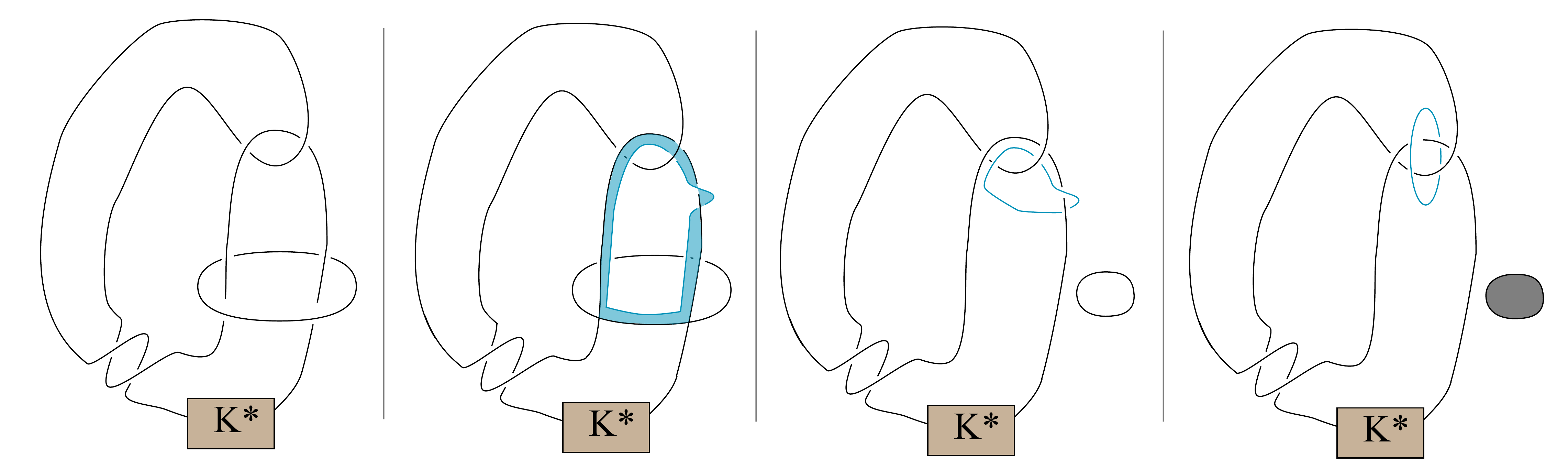}\\
\includegraphics[width=.5\linewidth]{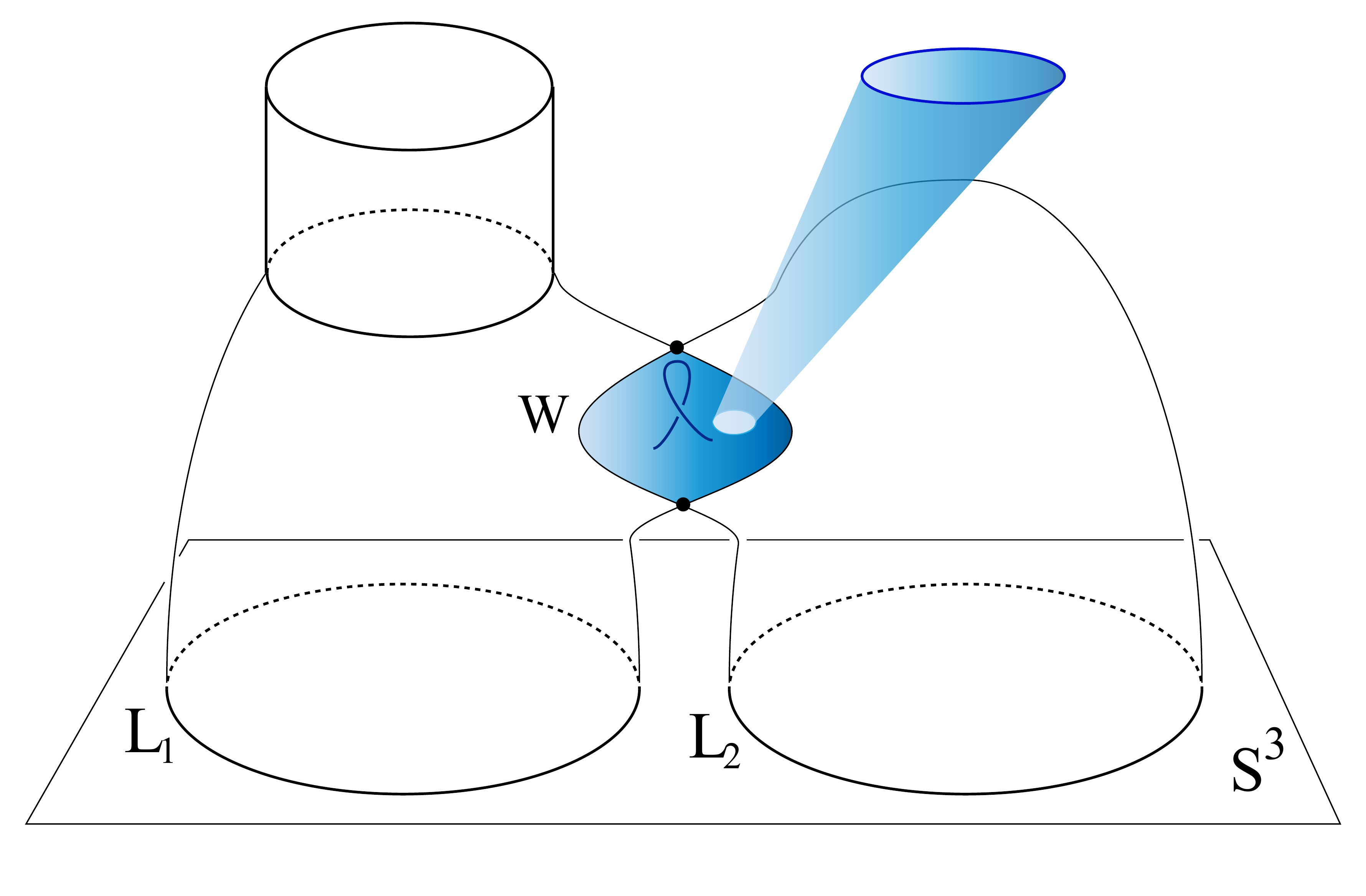}\\
\end{center}
\caption{One stage in the construction of a Cochran tower for the link in Example \ref{ex:infmany}. }\label{fig:infmany}
\end{figure}

Now we build a Cochran tower of arbitrarily high order for the link. First tie the mirror image of the trefoil (denoted $K^*$ in Figure~\ref{fig:infmany}) inside a small ball in the first component. This will not alter the Cochran invariants. 
Pulling the second component off the first component we get two intersections which are paired by a $+1$-twisted Whitney ``annulus".
The free boundary of this annulus is a meridian to the second band and hence the resulting link is isotopic to the original. This is schematically pictured at the bottom of Figure~\ref{fig:infmany}.

\begin{figure}[h]
\begin{center}
\includegraphics[width=\linewidth]{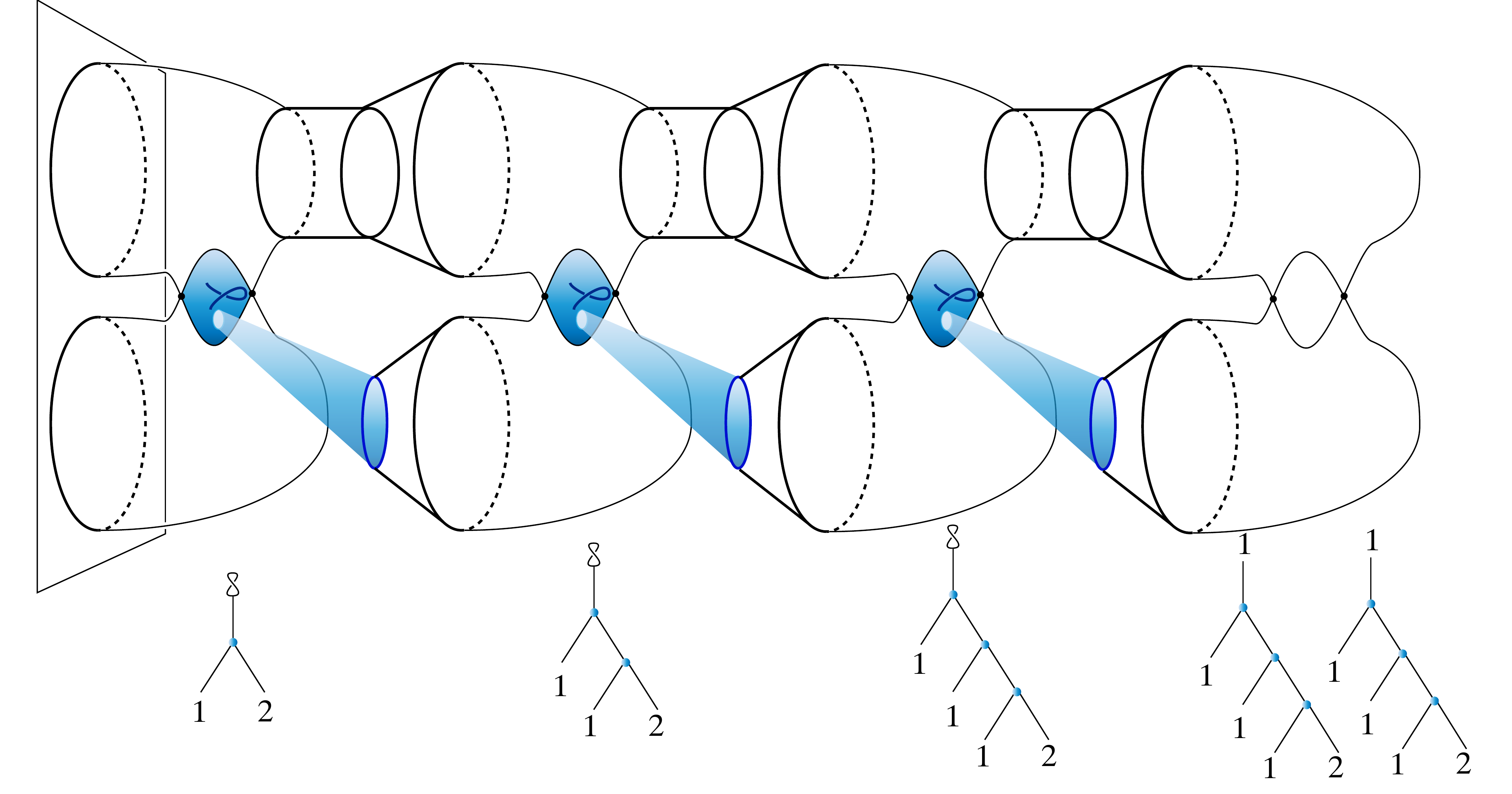}\\
\end{center}
\caption{Iterating the construction to get a Cochran tower of arbitrary order.}
\label{fig:iterating}
\end{figure}

One can then iterate this construction arbitrarily, concatenating this basic homotopy  repeatedly, and creating a single $+1$-twisted Whitney disk whose associated tree is $t^{\iinfty}_i$ at each $i$-th iteration. 
To get a twisted Whitney tower one eventually needs to \emph{not} pair the two intersection points, see Figure~\ref{fig:iterating} for the case of three iterations. The first component can then be capped off with a ribbon disk (because we added the inverse $K^*$). 
After $k$ iterations the resulting intersection forest $t(\cW)$ will contain single $\iinfty$-trees $t^{\iinfty}_i$ associated to the $+1$-twisted Whitney disks for each $i\leq k$, plus two order $k$ framed trees associated to the final two unpaired points at the top of the tower. These framed trees are $\beta$-bad, since they have a single $2$-label and $(k+1)$ $1$-labels, 
so we can only compute Cochran invariants up to a finite order using this tower. Iterating the construction $2k+1$ times would yield an order $2k$ Cochran tower $C$, with the $t^{\iinfty}_i$ trees in $C$ demonstrating that $\beta^i(L)=1$, for $i\leq k$. 

This also illustrates why the linear framed trees at the last step (Figure~\ref{fig:iterating}, lower right) must be considered $\beta$-bad.
\end{example}

\begin{example}\label{ex:no-2-labels-are-bad} (The twisted trees
$(1,1,\cdots,1)^\iinfty$ are $\beta$-bad.) We start with two illustrations of why $(1,1)^\iinfty=\tree{1}{1}{\iinfty}$ is $\beta$-bad. The disjoint union of a trefoil knot and an unknot bounds a twisted Whitney tower with the single intersection tree $\tree{1}{1}{\iinfty}$, and the Cochran invariants of this link are trivial.  
On the other hand, we claim that there is a link with a Whitney tower that has this same single intersection tree, but which has infinitely many nontrivial Cochran invariants. Thus the presence of a tree of this form contributes indeterminacies to the  tower which affect the Cochran invariants, and therefore $\tree{1}{1}{\iinfty}$ must be considered $\beta$-bad. The link $L$ is depicted in Figure~\ref{fig:bad}. 
\end{example}

\begin{figure}[h]
\begin{center}
\includegraphics[width=\linewidth]{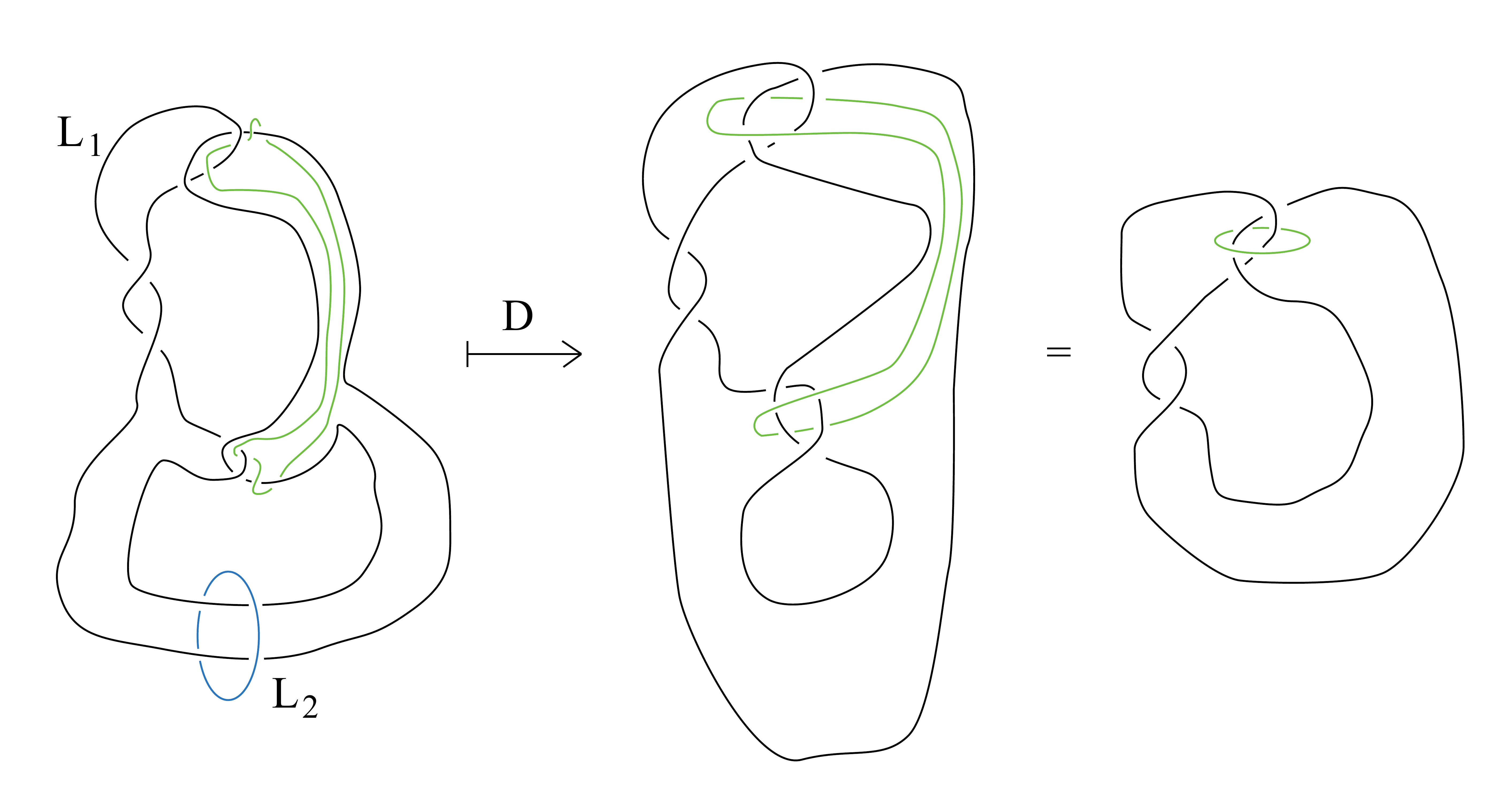}
\end{center}
\caption{A link $L=(L_1,L_2)$ with $\beta^i(L)\neq 0$ for all $i\geq 2$. $L$ bounds a twisted Whitney tower $\cW$ with $t(\cW)$ consisting of a single $\beta$-bad tree $\tree{1}{1}{\infty}$.} 
\label{fig:bad}
\end{figure}

The intersection knot (shown in green in Figure~\ref{fig:bad}) is untwisted, so $\beta^1(L)=0$; but $\beta^i(L)=1$ for all $i\geq 2$, since $D(L)$ is similar to the link with infinitely many Cochran invariants from Example~\ref{ex:infmany}: The only difference is that the second component of $D(L)$ is a meridian to the Seifert surface of a figure 8 knot instead of a trefoil, but the discussion of Example~\ref{ex:infmany} applies to this case, too.

A twisted Whitney tower $\cW$ bounded by $L$ with $t(\cW)$ consisting of a single $\beta$-bad tree $\tree{1}{1}{\infty}$ can be constructed as the trace of the null-homotopy of $L$ into $B^4$ which pulls apart the clasps of $L_1$, creating a pair of self-intersections in $W_1$ which admit an embedded $1$-twisted Whitney disk $W_{(1,1)}$. The interior of $W_{(1,1)}$, and the interior of $W_2$ bounded by $L_2$ are both free of intersections. 

Figure~\ref{fig:11infty-push-A} gives a general illustration how the presence of a twisted $W_{(1,1)}$ contributing $\tree{1}{1}{\iinfty}$ to $t(\cW)$ allows the creation of an arbitrary number of $t_i^\iinfty$-trees by manipulating $\cW$.

The same construction as in Figure~\ref{fig:11infty-push-A}, but with a twisted $W_{(1,1,\cdots,1)}$ in the place of $W_{(1,1)}$ (and with $J=2$), shows how the presence of 
$(1,1,\cdots,1)^\iinfty\in t(\cW)$ also allows the creation of an arbitrary number of $t_i^\iinfty$-trees.

\begin{figure}[h]
\begin{center}
\includegraphics[width=\linewidth]{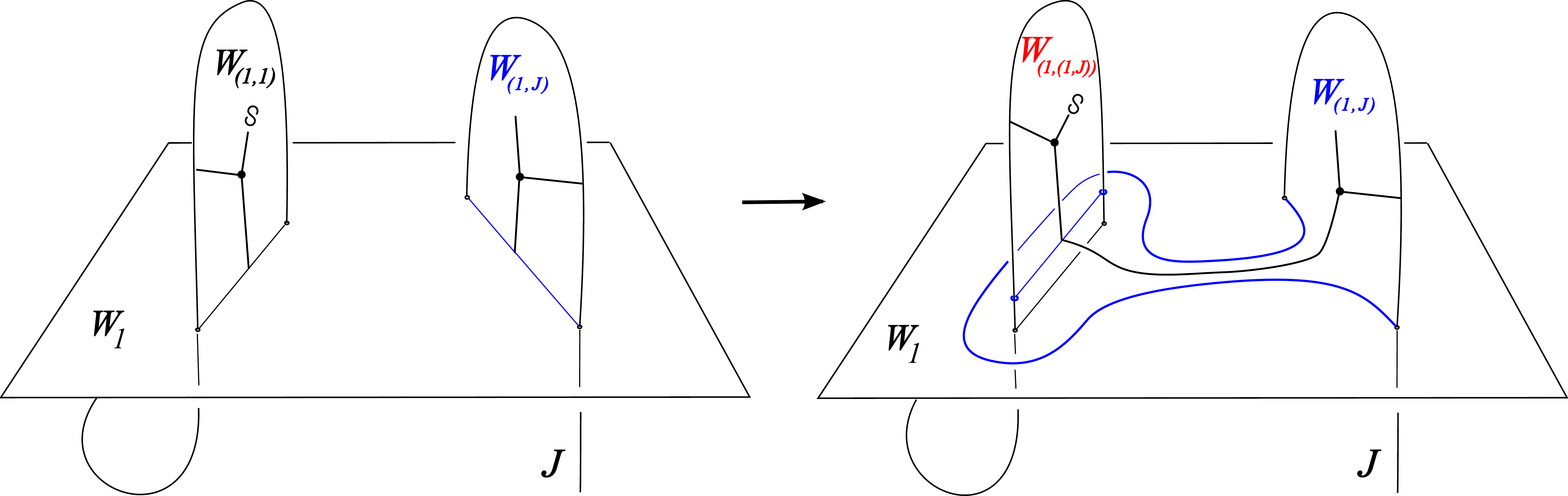}
\end{center}
\caption{Left: A twisted Whitney disk $W_{(1,1)}$, and a clean Whitney disk $W_{(1,J)}$ (which can be created for any $J$ by finger moves). Right: Pushing a collar of $W_{(1,J)}$ over $W_{(1,1)}$ gives rise to a new twisted Whitney disk $W_{(1,(1,J))}$ which could create twisted trees $t_i^\infty$ for appropriate $J$. For instance, $t_2^\infty$ could be created in the case $J=2$, and $t_3^\infty$ could be created in the case $J=(1,2)$.}
\label{fig:11infty-push-A}
\end{figure}


\section{Computing Cochran's $\beta^i$ via work of Kojima and Kirk}\label{sec:kirk}
After describing the relationship between the Cochran invariants and an invariant of Kojima, this section applies related computational techniques used by Kirk to give a family of examples which further illustrate how the presence of $\beta$-bad trees in $t(\cW)$ can lead to indeterminacies in the computation of the $\beta^i$ from $\cW$. This section is not used in the proofs of our main results, which are given in section~\ref{sec:proof}.

Cochran shows in \cite{Co} that his $\beta^i$ invariants are related to an invariant due to Kojima, called the $\eta$-invariant \cite{Kojima}. 
Consider the infinite cyclic cover of the complement of a knot $L_1$, denoted $Y$, and let $\Delta(t)$ be the symmetrized Alexander polynomial of $L_1$. 
For a link $L=(L_1,L_2)$ with linking number zero, let $z$ be a lift of $L_2$ to $Y$, $z_0$ a nearby lift of an untwisted parallel of $L_2$ and $t_*$ a generator of the covering transformation group. 

Then $\Delta(t_*)$ kills $z\in H_1(Y)$, so $\Delta(t_*)(z)=\partial d$ for some 2-chain $d$ in $Y$ and Kojima's $\eta$ function can be defined as 
\[
\eta(L)=\sum_{n=-\infty}^\infty\frac{1}{\Delta(t)}(z_0\cdot t_*^nd)t^n \in \mathbb Z[t,t^{-1}].
\]
 Cochran proves that under the change of variables $x=(1-t)(1-t^{-1})$, one has
 \[
 \eta(L)=\sum_{i=1}^\infty \beta^i(L)x^i  \in \Z[|x|].
 \]
 
An important consequence of this formula (with the Alexander polynomial in the denominator) is the existence of many links $L$ (with $\Delta(t) \neq 1$) for which $\beta^i(L)\neq 0$ for {\em infinitely many $i$}. 
 
In certain cases this can be related to Wall's self-intersection invariant via a procedure for constructing disjointly immersed surfaces from links used by Kirk to study link maps in the $4$--sphere \cite{Kirk}:
Consider the special case where  $L_1$ is unknotted, and think of the link $L$ as lying in the upper boundary of $S^3 \times I$. Because the linking number is zero, there is a homotopy of second component $L_2$, in the complement of $L_1$, that ends with the trivial link. Thinking of the track of this homotopy as an immersed annulus $A$ in the complement of $L_1 \times I$, 
one can consider the equivariant intersection number of $A$ with a parallel copy, $\lambda(A, A) \in \Z[t,t^{-1}]$. Both copies of $A$ are immersed annuli in the $4$--manifold $(S^3 \smallsetminus L_1)\times I \simeq S^1$ with fundamental group $\Z$, generated by $t$.

Note that setting $t=1$ gives zero because one gets the intersection number of an annulus (rel. boundary) in a manifold with no relative second homology. Moreover, the right hand side comes from a hermitian form and is hence invariant under the involution $a \mapsto \overline{a}$ on the group ring $\Z[t,t^{-1}]$ determined by $\overline{t} := t^{-1}$. 
It follows that $\lambda(A, A)$ is a polynomial in $x=(1-t)(1-t^{-1})$ and in fact Kirk showed in \cite{Kirk} that Kojima's $\eta$ function is given by
 \[
 \eta(L)=\lambda(A, A) \in \Z[t,t^{-1}]
 \]
This is the usual relation, in the equivariant setting, between the intersection form of a 4--manifold and its linking form on the boundary.

 To compute $\lambda(A, A)$, recall Wall's formula for the relation of the {\em self-intersection} invariant $\mu(A)$ with the intersection with a parallel copy. If the parallel copy has linking numbers zero on the boundaries then we get
 \[
 \lambda(A, A) = \mu(A) + \overline{\mu(A)}
 \] 
As a consequence, the easiest way to compute the invariants $\beta^i(L)$ for a link $L$ with linking number zero and unknotted component $L_1$ is as follows:  Do crossing changes on $L_2$ to separate it from $L_1$, and for each such crossing change record $\pm t^n$, where the sign is determined by the right hand rule for the crossing. Moreover, $n$ is the linking number with $L_1$ of the accessory circle of the crossing (which leaves on one sheet of $L_2$ and returns on the other).

Note that there are two choices for this accessory circle that differ by $L_2$ itself. Since it links $L_1$ trivially, this choice is irrelevant. However, the integer $n$ is only defined up to sign since we don't know in which direction we should run the accessory circle. This is the usual indeterminacy of Wall's self-intersection invariant $\mu$, and it disappears when computing the above average over the involution: One has to record $\pm (t^n + t^{-n})$ for each crossing change and sum over all necessary crossing changes.

This leads to a simple computation of
$$
\sum_{i=1}^\infty \beta^i(L)x^i  =  \lambda(A, A) = \mu(A) + \overline{\mu(A)}  \quad \text{ where } \quad x=(1-t)(1-t^{-1}).
$$
Note that in this case, only finitely many of the invariants $\beta^i(L)$ can be nontrivial! 

We also note that the exact same computation is valid if $L_1$ is allowed to be knotted but has trivial Alexander polynomial. By Freedman's Theorem \cite{FQ}, $L_1$ allows a topological concordance in $S^3 \times I$ to the unknot, with complementary fundamental group $\Z$. Therefore, $L_2$ bounds an immersed annulus $A$ in the complement of the concordance (leading to the trivial link) and we can again compute $\mu(A)$. 

The formulas above extend to this setting because $0$-surgery on $L_1$ is the boundary of the complement of (an open neighborhood of) the slice disk in $D^4$. Moreover, Kojima's function is the boundary value of the $\Z$-equivariant linking form associated to the intersection form on this $4$--manifold.

\begin{example}\label{ex:t2-bad} (The tree $t_2$ is $\beta$-bad.) This example uses clasper-surgery to illustrate in detail how the presence of the framed tree $t_2\in t(\cW)$ leads to indeterminacies in the computation of higher-order Cochran invariants in several orders at once. As described in Theorem~\ref{thm:clasper-concordance} (Section~\ref{sec:proof} and \cite{CST6}), the tree-types of clasper surgeries on an unlink correspond to trees in the intersection forest of a Whitney tower bounded by the resulting link. 

Consider the clasper (see Section~\ref{subsec:clasper-conventions}) of type $t_2$ pictured below on the left: \
\begin{center}
\includegraphics[width=.25\linewidth]{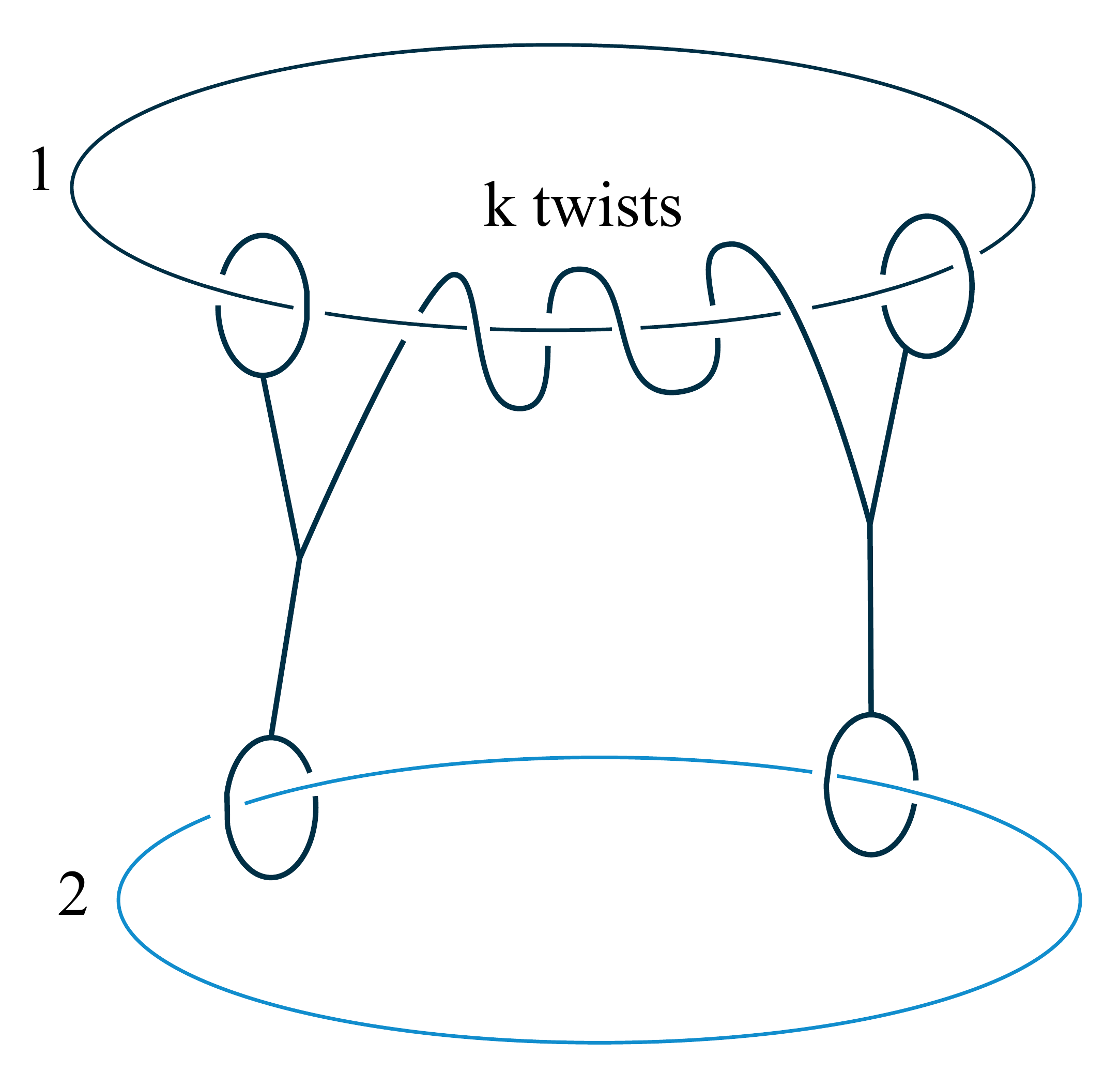}$\hfill$\includegraphics[width=.7\linewidth]{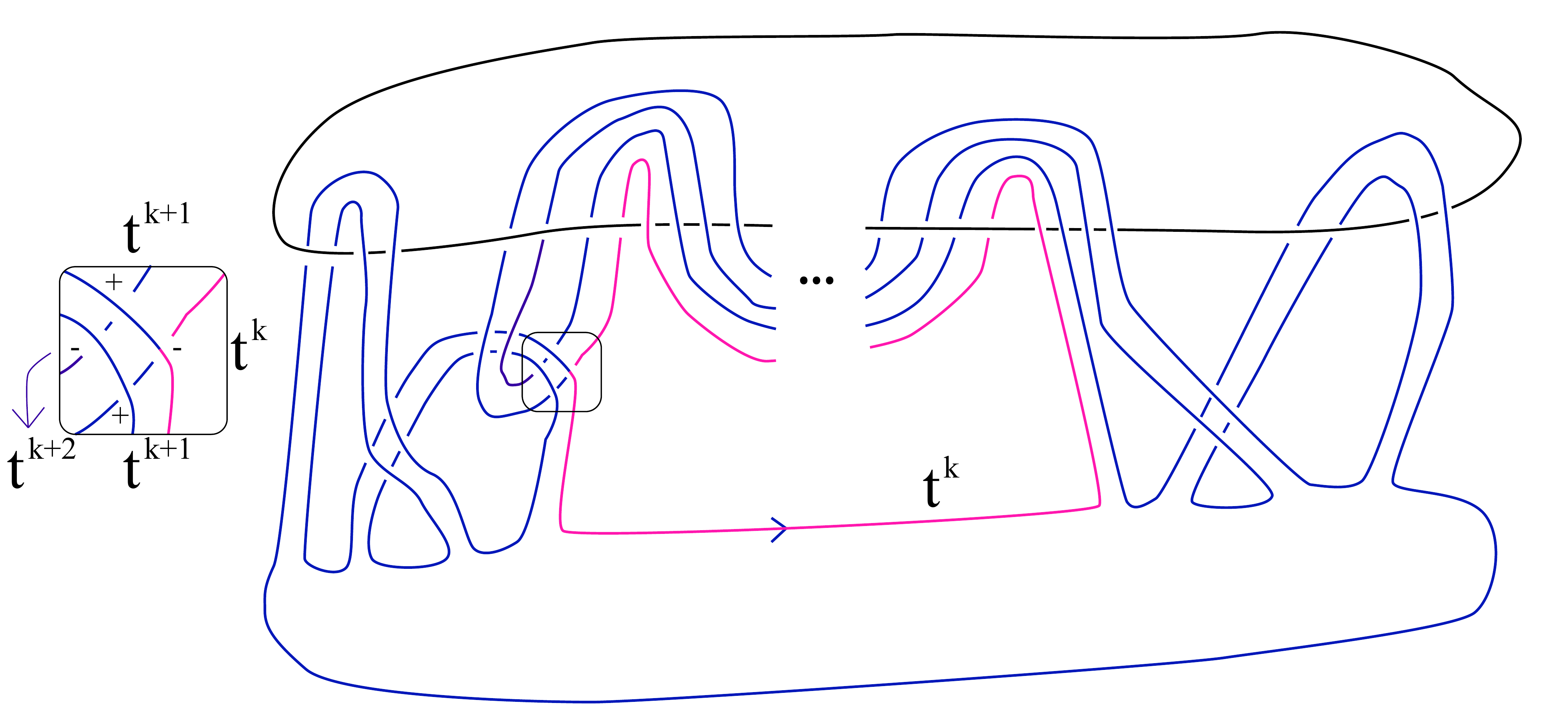}
\end{center}
Let $L^k$ be the link obtained from this surgery, pictured on the right. We calculate $\eta(L^k)$ via the Wall self-intersection invariant
using the procedure just outlined. There are four crossing changes of component $2$, pictured in the small box, which will turn $L^k$ into an unlink. The accessory arc for one of these crossings is colored pink, and one can see that it links with component one $k$ times, and so is assigned $-t^k$, the negative coefficient coming from the fact that the crossing is a negative crossing. One similarly calculates the accessory circles for the other three crossings getting $-t^{k+2}, t^{k+1}, t^{k+1}$. Thus
$\eta(L^k)=-(t^k+t^{-k})+2(t^{k+1}+t^{-k-1})-(t^{k+2}+t^{-k-2})$.  Converting this to a power series in $x=(1-t)(1-t^{-1})$ we have the following table of polynomials in $x$:
\begin{center}
\begin{tabular}{r|l}
$k$&$\sum_{i=1}^\infty \beta^i(L^k)x^i$\\
\hline
$-1$& $2x$\\
$0$& $2x-x^2$\\
$1$& $2 x - 4 x^2 + x^3$\\
$2$& $2 x - 9 x^2 + 6 x^3 - x^4$\\
$3$& $2 x - 16 x^2 + 20 x^3 - 8 x^4 + x^5$\\
$4$& $2 x - 25 x^2 + 50 x^3 - 35 x^4 + 10 x^5 - x^6$
\end{tabular}
\end{center}
The fact that $\beta^1(L^k)=2$ for all these examples comes from the fact that the link bounds a Whitney tower with the single $\beta$-bad tree $t_2$, and this contributes $2$ to the Sato-Levine invariant $\mu_{1122}$. On the other hand, the higher $\beta^i(L^k)$ are not constant on this class of links, so surgery on a $t_2$-clasper will not have a predictable effect on the Cochran invariants.

\end{example}

\section{Proof of Theorem~\ref{thm:main}}\label{sec:proof}

After fixing some terminology and notation, this section will prove Theorem~\ref{thm:main} by combining Habiro's clasper surgery techniques \cite{Hab00} with the Whitney tower obstruction theory \cite{CST} and the following decomposition theorem from \cite{CST6}:

\begin{theorem}\label{thm:clasper-concordance}
If a link $L$ bounds a (twisted) Whitney tower $\cW$ then there is a finite sequence of concordances and simple (twisted) clasper surgeries from the unlink to $L$ such that
the tree-types of the clasper surgeries correspond to the trees in $t(\cW)$, with the twistings on twisted claspers corresponding to the twistings on the twisted Whitney disks of $\cW$. $\hfill\square$
\end{theorem}
A more precise description of the relationship between clasper concordance and Whitney towers is given in \cite{CST6}. 

See section~\ref{subsec:clasper-conventions} below for clasper terminology and conventions, including tree-types and twisted claspers.

\subsection{Outline of proof of Theorem~\ref{thm:main}}\label{subsec:proof-outline}

The proof of the second statement of Theorem~\ref{thm:main}, that Cochran invariants up to order $k$ can be calculated from the intersection forest of a Cochran tower of order $2k$, will proceed as follows:
By Theorem~\ref{thm:clasper-concordance}, the Cochran tower implies the existence of a sequence of clasper surgeries and concordances from the unlink $U$ to the link $L$, where the tree-types and twistings of the claspers are the same as those in the intersection forest of the Cochran tower. Since the Cochran invariants are concordance invariants, we then observe how they change under clasper surgeries:
On the one hand, a
$t^{\iinfty}_i$-surgery will be shown to leave $\beta^j$ unchanged for $ j\neq i$, and to change $\beta^i$ by the twisting 
$\omega(\Gamma)$ of the twisted $t^{\iinfty}_i$-clasper $\Gamma$ in Corollary~\ref{cor:gamma-t}. On the other hand,
clasper surgeries corresponding to all other trees allowed in an order $2k$ Cochran tower will be shown to preserve the $\beta^i$ for all $i\leq k$ in Corollary~\ref{cor:three-twos} and Propositions \ref{prop:framed-zero-trees}, \ref{prop:twisted-zero-trees}, and \ref{prop:above-order}.  So $\beta^k(L)-\beta^k(U)$ will be the weighted count of $t^{\iinfty}_i$ trees as claimed.

The proof of the first statement of Theorem~\ref{thm:main}, that a link $L$ bounds a Cochran tower of order $2k$ in $B^4$ if and only if $L$ has vanishing linking number and the Arf invariant of $L_1$ vanishes, will follow from the fact that the $\beta$-bad trees in the intersection forest of a Whitney tower bounded by $L$ can all be eliminated by twisting constructions or geometric IHX-constructions, at the cost of only creating higher-order trees \cite{CST}. 
See section~\ref{subsec:proof-theorem-main}.

\subsection{Clasper conventions}\label{subsec:clasper-conventions}

We will be using Habiro's {\em clasper surgery} techniques, for details see \cite{CT1,Hab00}. We adopt the following terminology from \cite{CT1}, and essentially the same notion of twisted claspers from \cite{CST3}.
Although claspers are surfaces, we follow the customary identification of a clasper with its $1$-dimensional spine, which is a framed unitrivalent graph.  

All of our claspers will be \emph{tree claspers}.
A clasper $\Gamma\subset S^3$ is \emph{capped} if all the leaves of $\Gamma$ bound disjointly embedded disks (the \emph{caps}) into $S^3\setminus \Gamma$ (so the leaves of a capped clasper are unknotted). A cap for a clasper on a link is called a \emph{simple cap} if it is $0$-framed and intersects the link in a single point.
A \emph{simple tree clasper} is a capped tree clasper on a link such that each cap is simple.  A \emph{twisted tree clasper} is a capped tree clasper on a link such that all caps are simple except for one $k$-framed cap, for some integer $k\neq 0$, whose interior is disjoint from the link. If $\Gamma$ is a twisted tree clasper, this \emph{twisting number} $k$ of $\Gamma$ will be denoted $\omega(\Gamma)$.

A leaf is \emph{simple} if it bounds a simple cap and \emph{twisted} (\emph{$k$-twisted}) if it bounds a twisted cap.
If a clasper leaf bounds a $0$-framed cap which is disjoint from all link components, then the leaf and the cap are said to be \emph{clean}.  

By the \emph{interior} of a clasper, we will mean the complement of the leaves (and caps), which in our case is a unitrivalent tree.
For simple and twisted tree claspers this tree is labeled exactly like the intersection trees in $t(\cW)$: Univalent vertices are labeled by the components that the corresponding caps intersect, or by the $\iinfty$ symbol if that leaf is twisted. 


A clasper with tree type $T$ will be called a \emph{$T$-clasper}. In particular, if $T$ is a twisted tree, then a $T$-clasper is a twisted clasper. 
 

{\bf Surgery with a specified root:} For $\Gamma$ a $T$-clasper in the complement of a link $L$, we denote by $L_\Gamma$ the link obtained by surgery on $\Gamma$. This surgery is really a diffeomorphism of $S^3$ which can change the link's isotopy class. The diffeomorphism and concomitant modification of $L$ can be realized in many different ways. For example, if one designates a univalent vertex of $T$ as the root, one can realize $L_\Gamma$ as a modification of $L$ which takes place in a neighborhood of $\Gamma$ and the root cap; the other caps are not involved. More precisely, only strands of 
$L$ which intersect the root cap will be modified, and this modification will be supported near $\Gamma$.
A surgery on a $T$-clasper will be called a \emph{$T$-surgery}. 


{\bf The Zip construction:} Suppose a leaf of a clasper $\Gamma$ bounds a cap $\Delta$ which may intersect the link $L$. The \emph{zip construction} \cite{CT1,Hab00} cuts the cap $\Delta$ into two pieces, such that surgery on $\Gamma$ is equivalent to surgery on the union of two daughter claspers, both of the same tree type as $\Gamma$, each of which inherits one piece of $\Delta$ as a cap, while the other leaves are parallel copies of the originals. The two daughter claspers are embedded in a neighborhood of $\Gamma$; their edges are disjoint from any caps that each may have.


\subsection{The Cochran $\beta^i$-invariants}\label{subsec:Cochran-invts-definition}
We focus on the setting of classical $2$-component links modulo concordance, and refer the reader to \cite{Co,Co1} for details on the full generality of the Cochran $\beta^i$-invariants. 

\begin{definition}[\cite{Co}]\label{def:derived-link}
Let $L=(L_1,L_2)$ be a $2$-component link with zero linking number, and let $\Sigma_1$ and $\Sigma_2$ be Seifert surfaces for $L_1$ and $L_2$ such that $\Sigma_1\cap\Sigma_2=\chi$ is connected (and non-empty). 
Then the \emph{derived link} $D(L)$ is defined to be the 2-component link $(L_1,\chi)$ gotten by replacing $L_2$ with $\chi$.
The knot $\chi$ is called the \emph{characteristic curve} of $L$ (and $\Sigma_1$ and $\Sigma_2$). 
\end{definition}

It is easy to see that a derived link $D(L)$ also has zero linking number, giving rise to the following infinite sequence of invariants, which Cochran showed to be concordance invariants of $L$:
\begin{definition}[\emph{Cochran invariants} \cite{Co}]\label{def:beta}
The first Cochran invariant $\beta^1(L)\in\Z$ is defined to be  the classical Sato-Levine invariant of $L$. Inductively, $\beta^i(L)$ is defined to be $\beta^{i-1}(D(L))$ for $i>1$. 
\end{definition}
The classical Sato-Levine invariant \cite{Sa} is equal to the {\em twisting} $\omega(\chi)\in\Z$ of the framing of the characteristic curve $\chi=\Sigma_1\cap \Sigma_2$ induced by the normal framing of either surface $\Sigma_j$ relative to the $0$-framing of $\chi\subset S^3$.

\begin{remark}\label{rem:connect}
As shown by Cochran \cite{Co}, for any $L=(L_1,L_2)$ with trivial linking number, there exist Seifert surfaces $\Sigma_1$ and $\Sigma_2$ that intersect in a connected closed curve. 
In the sequel, it will often happen that modification of such a link will temporarily create new components of intersection between the modified Seifert surfaces, and we will use the following procedure for making the intersection connected in a well-controlled way: The modification of $L$ will lead to $\Sigma_1$ being extended by boundary-summing into some new genus one surfaces, and the corresponding new intersection curves with the second Seifert surface will be at most a single nonseparating circle on each of these new genus one subsurfaces.
Choose arcs in the new $\Sigma_1$ connecting each new intersection circle to the original characteristic curve $\chi=\Sigma_1\cap\Sigma_2$. Surgering the second Seifert surface along these arcs has the effect of band-summing the new intersection curves with $\chi$, making the intersection between the new Seifert surfaces connected. Since each of the new circles was nonseparating, these arcs can be chosen in such a way that the surgered surface is still orientable. 
 (For example, see Figures~\ref{fig:main-induction}, \ref{fig:twisted-zero-trees} and \ref{fig:1122tree}, where new genus is added to $\Sigma_1$ along with new intersection components supported in the new genus.)

%
\end{remark}

\subsection{Claspers and Cochran invariants}\label{subsec:claspers-Cochran-invts}
It turns out that the Cochran invariants are \emph{finite type} in a certain sense, ultimately deriving from the fact that the twisting $\omega(K)$ of a framed knot $K$ is a type $1$ framed knot invariant, a notion we now define.

Recall that a nowhere-vanishing normal vector field of a smooth knot $K\subset S^3$ is referred to as a {\em framing} of $K$. Identifying the image of the vector field with a knot parallel to $K$, the twisting $\omega(K)$ is the linking number of $K$ with its parallel copy. Here the notation $K$ is abused to describe the {\em framed knot}. 

\begin{definition}\cite{CDM}
A \emph{singular framed knot} (in $S^3$) is a framed knot, modified to allow finitely many double points or non-tangential zeros of the framing. Any framed knot invariant $v$ can be extended to singular framed knots, by the usual Vassiliev skein relation for double points and by
\begin{center}
\includegraphics[width=4cm]{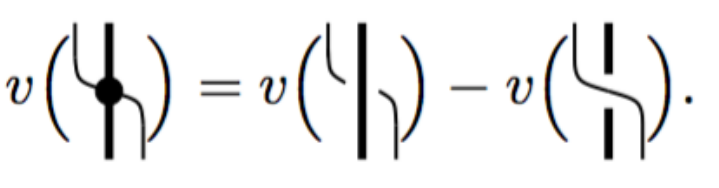}
\end{center}
for the non-tangential zeros. A  framed knot invariant is said to be \emph{type $k$} if it vanishes on singular framed knots with $k+1$ double points and non-tangential zeroes.
\end{definition}

In other words, for any collection $S$ of $k+1$ disjointly supported crossing changes and local framing changes, a type $k$ framed knot invariant $v$ will satisfy $$\sum_{S'\subset S} (-1)^{|S'|}v(K_{S'})=0$$ 
where $K_{S'}$ the result of changing $K$ by exactly the elements in $S'$, $|S'|$ is the cardinality of $S'$, and the sum is over all subsets of $S$. Crossing changes and local framing changes are both examples of homotopies of $K$ supported in balls (where we consider the local framing change to be a crossing change between $K$ and the parallel knot corresponding to the framing), and all homotopies and framing changes can be represented as a sequence of these two basic moves (together with isotopy.)
As observed by Goussarov \cite{goussarov}, an alternating sum over a collection $S$ of groups of crossing and framing changes (all disjointly supported) can be written as a linear combination over alternating sums of crossing and framing changes. Thus, we have that an invariant is of type $k$ if and only if it vanishes on alternating sums of $k+1$ disjointly supported homotopies.

It is not difficult to show that the twisting $\omega(K)$ is a framed knot invariant of type $1$. This has implications for the Cochran invariants since we ultimately compute them using the twisting of the characteristic curve. However, the $\beta^i$ are not quite type $1$ invariants, but rather are only type $1$ invariants if one only allows crossing changes on component $2$. This leads to the following definition:

\begin{definition}
A link invariant $f$ is \emph{type $k$ with respect to a given component} if, when given 
any set $S$ of any $k+1$ disjointly supported crossing changes of this component with itself, we have
$$\sum_{S'\subset S}(-1)^{|S'|} f(L_{S'})=0,$$
where $L_{S'}$ denotes the result of changing $L$ by exactly the homotopies in $S'$, $|S'|$ is the cardinality of $S'$, and the sum is over all subsets of $S$. 
\end{definition}

\begin{remark}\label{rem:fintype} 
As above, we can replace crossing changes with arbitrary homotopies of the given component supported in disjoint balls that avoid the other components. 
\end{remark}

\begin{remark}\label{rem:framings} 
In the above paragraphs, we introduced two distinct notions of finite type invariants, one for framed knots and one for links with a specified component. There is an obvious way to combine those notions, namely for links with a {\em specified framed} component. Cochran's derived link is exactly of that type and Proposition~\ref{prop:fintype} below holds for all $i\geq 0$ if we define $\beta^0(L):=\omega(L_2)$ for 2-component links with a framing on the specified component $L_2$. Then the proof below is an induction that naturally extends all the way down to $i=0$. 
\end{remark}

\textbf{Convention:} For the rest of the paper we assume that $L$ designates a link for which the $\beta^i$-invariants are defined, and that $L$ comes equipped with Seifert surfaces which intersect in a single circle.

\begin{proposition}\label{prop:fintype}
For all $i\geq 1$, the Cochran $\beta^i$-invariants of $L=(L_1,L_2)$ are type $1$ with respect to $L_2$.
\end{proposition}
\begin{proof}
We start by showing that the first Cochran invariant $\beta^1$ vanishes on an alternating sum over a pair of disjointly supported crossing changes on $L_2$. These crossing changes can be realized as a pair of edge-clasper surgeries $\Gamma_j$, $j=1,2$, where $\Gamma_j$ has two leaves, each of which is a meridian to $L_2$ and bounds a cap that intersects $\Sigma_2$ in a clasp singularity. We may further assume these caps are disjoint from $\Sigma_1$. 

Now the single edge of each $\Gamma_j$ will crash through sheets of both Seifert surfaces (Figure~\ref{fig:type1}, left). Pushing these sheets off the edge of $\Gamma_j$ will create ribbon singularities with $\Sigma_2$. When $\Sigma_2$ is pushed off, these new singularities can be resolved using the standard ribbon singularity resolution, increasing the genus of $\Sigma_2$ (Figure~\ref{fig:type1}, second from left, blue). 

When $\Sigma_1$ is pushed off, pairs of algebraically cancelling intersections between $\Sigma_1$ and $L_2$ will be created. Resolve these by adding tubes in a neighborhood of $L_2$ (Figure~\ref{fig:type1}, second from left, green). This will create new circles of intersections between $\Sigma_1$ and $\Sigma_2$ that travel through one leaf of each $\Gamma_j$. Apply the procedure of Remark~\ref{rem:connect} (as needed) to modify the surfaces in a neighborhood of $\Sigma_1\cup\Sigma_2$ so that $\Sigma_1\cap\Sigma_2$ becomes connected. At this point we have modified the surfaces so that they are disjoint from the interior of the claspers, but one leaf of each clasper may link the characteristic curve $\chi(L)$ of $L$. 
Note that we have  pushed the surface sheets across the same leaf in each clasper so the other leaf can be used as a root for surgery.
\begin{figure}
\includegraphics[width=\linewidth]{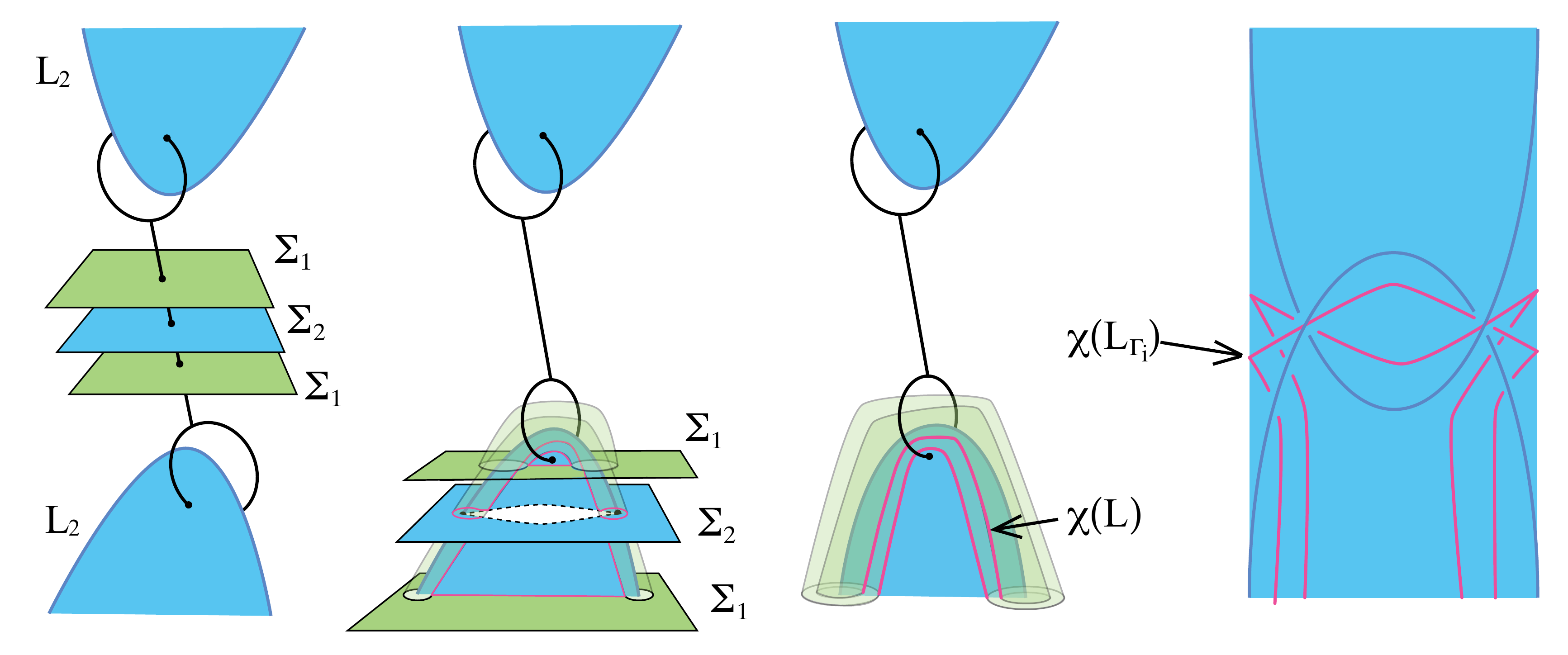}
\caption{Left: An edge-clasper $\Gamma_j$ on $L_2$ which has two interior intersections with $\Sigma_1$ and one interior intersection with $\Sigma_2$. Second from left: Pushing the surface sheets off the edge creates more genus and intersections among the surfaces. Third from left: Zooming in near the leaf, nested tubes from $\Sigma_1$ containing arcs of the characteristic curve $\chi(L)$ (red) of $L$ also link this leaf. Right: A local picture of the characteristic curve $\chi(L_{\Gamma_j})$ (red) for $L_{\Gamma_j}$ near the clasp created by $\Gamma_j$-surgery.
}
\label{fig:type1}
\end{figure}

Now we consider the new characteristic curves $\chi(L_{\Gamma_j})$ for the $L_{\Gamma_j}$. Each $L_{\Gamma_j}$ has an added clasp, and $\Sigma_2$ can be locally modified into a Seifert surface for component 2 of $L_{\Gamma_j}$, which locally looks like a disk with a twisted band attached, as in the far right hand side of Figure~\ref{fig:type1}. The tubes from $\Sigma_1$ consist of meridians to $L_2$ which now travel over this band. So the characteristic curve for $L_{\Gamma_j}$ is formed by a homotopy $H_j$ of $\chi(L)$ in a neighborhood of $\Gamma_j$ which will introduce the local twisting and crossing changes illustrated in the right-most picture of Figure~\ref{fig:type1}. Specifically, for each tube of $\Sigma_1$ passing through the leaf of $\Gamma_j$ the homotopy $H_j$ will create one twist in the characteristic curve framing, and crossing changes with the other arcs in tubes through the same leaf.

Since the twisting $\omega$ is a type 1 framed knot invariant, this shows that $\beta^1(L)$ is type $1$ with respect to $L_2$:
\begin{align*}
  \beta^1(L)&-\beta^1(L_{\Gamma_1}) -\beta^1(L_{\Gamma_2})+-\beta^1(L_{\Gamma_1\cup\Gamma_2})  \\
   &=\omega(\chi(L))-\omega(\chi(L_{\Gamma_1}))-\omega(\chi(L_{\Gamma_2}))+\omega(\chi  (L_{\Gamma_1\cup\Gamma_2})) \\
   &=\omega(\chi(L))-\omega(\chi(L)_{H_1}))-\omega(\chi(L)_{H_2}))+\omega(\chi  (L)_{H_1\cup H_2}))=0.
\end{align*}

Now to see that $\beta^i(L)$ is type 1 with respect $L_2$ for all $i>1$, observe that by the above constructions
the $H_j$ are disjointly supported homotopies of the second component of $D(L)$. So $D(L_{\Gamma_j})=D(L)_{H_j}$, and by induction
\begin{align*}
\beta^i(L)&-\beta^i(L_{\Gamma_1})-\beta^i(L_{\Gamma_2})+\beta^i(L_{\Gamma_1\cup\Gamma_2}) \\
=&\beta^{i-1}(D(L))-\beta^{i-1}(D(L)_{H_1})-\beta^{i-1}(D(L)_{H_2})+\beta^{i-1}(D(L)_{H_1 \cup H_2})=0.
\end{align*}

\end{proof}

\begin{corollary}\label{cor:three-twos} 
If $T$ is any framed tree having at least three $2$-labels, or any twisted tree having at least two $2$-labels, and 
$\Gamma$ is any $T$-clasper on $L$, then $\beta^i(L_{\Gamma})=\beta^i(L)$ for all $i\geq 1$.
\end{corollary}
\begin{proof}
We realize $L_\Gamma$
by doing surgery, using one of the $2$-labeled leaves of $\Gamma$ as the root. Now $L_\Gamma$ has strands of component $2$ (but not component $1$) going through a neighborhood of $\Gamma$.

Given a non-root simple leaf $\ell$ of $\Gamma$ whose cap $\Delta$ intersected $L_2$, in $L_\Gamma$ there are now multiple parallel strands of component $2$ running through where $\ell$ used to be. Let $H_\ell$ be a homotopy that pushes the strand of component $2$ that intersected $\Delta$ across the parallel strands. 
Then $(L_{\Gamma})_{H_\ell}$ is actually surgery of $L$ along a clasper that has a clean cap. This implies that ${(L_{\Gamma})_{H_\ell}}=L$. 

Given a non-root twisted leaf $\ell^\iinfty$ of $\Gamma$, let $H_{\ell^\iinfty}$ be a homotopy that straightens the strands of component $2$ traveling through where ${\ell^\iinfty}$ used to be. By similar reasoning as in the previous paragraph $(L_{\Gamma})_{H_{\ell^\iinfty}}=L$. 

Moreover, given several of these two types of homotopies supported near different leaves of the clasper, modifying $L_\Gamma$ by any nonempty subset will yield $L$, as it is equivalent to making more than one leaf bound a clean cap in the complement of both components.
The hypotheses of the corollary imply that we may find two disjointly supported homotopies of this form. Thus, by the Proposition~\ref{prop:fintype}: $\beta^i(L_{\Gamma})-\beta^i(L)-\beta^i(L)+\beta^i(L)=0$.
\end{proof}

Recall that $t^{\iinfty}_k$ is the order $k$ twisted tree having $k$ $1$-labels, one $2$-label and one $\iinfty$-label as in Figure~\ref{fig:t-tree}.

\begin{proposition}
If  $\Gamma$ is a $t^{\iinfty}_1$-clasper on $L$, then $L_{\Gamma}=B_{\omega(\Gamma)}\#_b L$, where $B_\omega$ is the $\omega$-twisted Bing double of the unknot, and $\#_b$ denotes a component-wise band-sum.
\end{proposition}
Here the bands guiding the band-sum run along the two edges of $\Gamma$ whose caps intersect $L$. 
\begin{proof}
Here is a picture of surgery on a $t^{\iinfty}_1$-clasper with twisting $\omega$.
\begin{center}
\includegraphics[width=6cm]{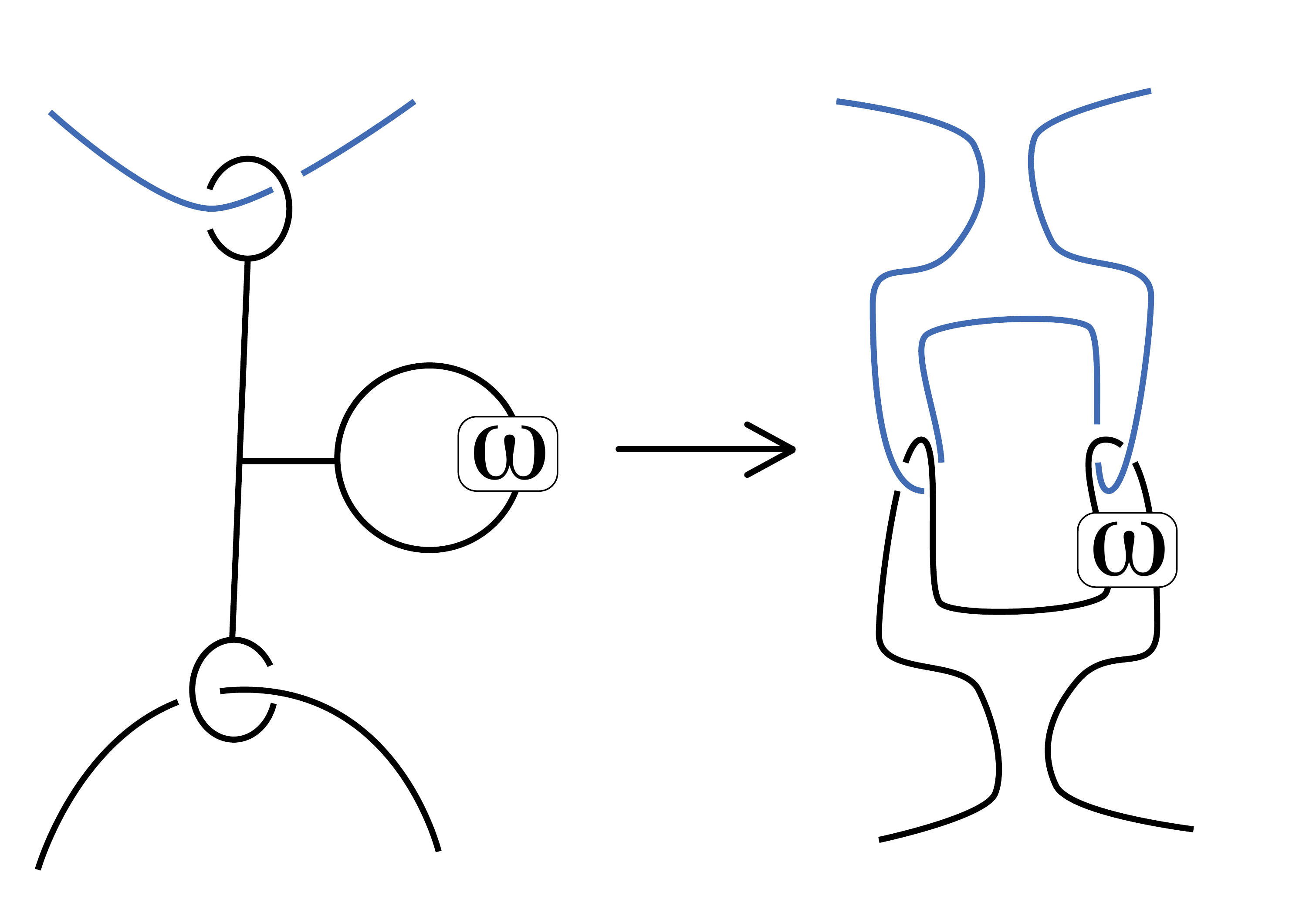}
\end{center}
This is exactly a band sum with $B_{\omega(\Gamma)}$ as claimed.
\end{proof}

\begin{corollary} \label{cor:epsilon}
Let $\Gamma$ be a $t_1^{\iinfty}$-clasper on $L$.
If $i\geq 2$, then $\beta^i(L)=\beta^i(L_{\Gamma})$, and $\beta^1(L_{\Gamma})=\omega(\Gamma)+\beta^1(L)$.
\end{corollary}
\begin{proof}
The invariants $\beta^i$ are additive under arbitrary band sum \cite[Thm.5.6]{Co}, and their value on $B_\omega$ are $0$ for $i\geq 2$ and $\omega$ for $i=1$.
\end{proof}

\begin{lemma}\label{lem:clasper-off-seifert-surfaces}
Let $T$ be a twisted tree with only one $2$-label. Then for any $T$-clasper $\Gamma$ on $L$, there exists a $T$-clasper $\Gamma'$ on $L$ satisfying the following:
\begin{enumerate}
\item
 $\beta^i(L_\Gamma)=\beta^i(L_{\Gamma'})$ for all $i\geq 1$. 
 \item
 $\omega(\Gamma')=\omega(\Gamma)$.
 \item\label{lem-item:clasper-disjoint-from-seifert}
The interior of $\Gamma'$ is disjoint from possibly new Seifert surfaces for $L$, and taking the derived link $D(L)$ with respect to the new Seifert surfaces is the same as the derived link taken with respect to the original Seifert surfaces. 
\end{enumerate}
\end{lemma}

\begin{proof}
Starting with the original Seifert surfaces $\Sigma_1$, $\Sigma_2$ for $L$, we may assume that each $j$-labeled cap of $\Gamma$ meets  $\Sigma_j$ in a clasp singularity, but is otherwise disjoint from both Seifert surfaces. 
First, we find a new $T$-clasper $\Gamma'$ such that the interior of $\Gamma'$ is disjoint from $\Sigma_2$, and $L_{\Gamma'}$ has the same Cochran invariants as $L_\Gamma$. 
If the interior of $\Gamma$ is not already disjoint from $\Sigma_2$, then do surgery on $\Gamma$, using the $2$-labeled leaf as the root. Then we can create the following two disjointly supported homotopies of the second component of $L_{\Gamma}$. The homotopy $H_1$ untwists the twisted strands of component $2$ near what was the twisted leaf of $\Gamma$. The homotopy $H_2$ pushes $\Gamma$ across the boundary of $\Sigma_2$, removing all intersections with $\Sigma_2$. (Here the $H_2$ is really pushing strands of component $2$ that are contained in a neighborhood of where $\Gamma$ was before the surgery.) Now $\beta^i(L_{\Gamma})-\beta^i((L_{\Gamma})_{H_1})-
\beta^i((L_{\Gamma})_{H_2})+\beta^i((L_{\Gamma})_{H_1\cup H_2})=0$. But $(L_{\Gamma})_{H_1}=L=(L_{\Gamma})_{H_1\cup H_2}$, so we get that $\beta^i(L_{\Gamma})=\beta^i((L_{\Gamma})_{H_2})$.
Now $(L_{\Gamma})_{H_2}=L_{\Gamma'}$, where $\Gamma'$ is a $T$-clasper (the result of applying $H_2$ to 
$\Gamma$) that has no interior intersections with $\Sigma_2$. 

So far we have produced a $T$-clasper $\Gamma'$ on $L$ satisfying $\forall i\,\,\beta^i(L_{\Gamma})=\beta^i(L_{\Gamma'})$ which does not have interior intersections with $\Sigma_2$. 
If the interior of $\Gamma'$ is also disjoint from $\Sigma_1$ we are done.

Otherwise, suppose that the interior of $\Gamma'$ is not disjoint from $\Sigma_1$. Since only one leaf of $\Gamma'$ links $L_2$, we can push all of these interior intersections 
away from the $2$-labeled leaf and toward the $1$-labeled leaves and twisted leaf, eventually pushing sheets of $\Sigma_1$ off of these leaves and resolving the new self-intersections of $\Sigma_1$ via the standard ribbon singularity resolution, as in the proof of Proposition~\ref{prop:fintype}. 
This modification of $\Sigma_1$ does not affect the original characteristic curve $\chi$. Moreover, $L$ and $\Gamma'$ are unchanged by this procedure. So $\forall i\,\,\beta^i(L_{\Gamma})=\beta^i(L_{\Gamma'})$, and we have removed all intersections between the interior of $\Gamma'$ and both Seifert surfaces.
\end{proof}

\begin{proposition}\label{prop:induct}
Let $\Gamma$ be a $t^{\iinfty}_k$-clasper on $L$ for $k\geq 2$. Then 
there exists a $t^{\iinfty}_k$-clasper $\Gamma'$ on $L$ such that  $\beta^i(L_{\Gamma})=\beta^i(L_{\Gamma'})$, for all $i\geq 1$, and
a $t^{\iinfty}_{k-1}$-clasper ${\Gamma''}$ on $D(L)$ such that 
$D(L_{\Gamma'})=D(L)_{\Gamma''}$ and $\omega(\Gamma)=\omega(\Gamma'')$.
\end{proposition}
\begin{proof}
Take $\Gamma'$ to be the clasper guaranteed by Lemma~\ref{lem:clasper-off-seifert-surfaces},
with the interior of $\Gamma'$ disjoint from both Seifert surfaces for $L$, as in the upper left of Figure~\ref{fig:main-induction}.
The clasper calculus and constructions illustrated in Figure~\ref{fig:main-induction} show how splitting off a Y-clasper from $\Gamma'$ yields the desired $t^{\iinfty}_{k-1}$-clasper ${\Gamma''}$ on $D(L)$.
\end{proof}

We get the following corollary of Proposition~\ref{prop:induct} and Corollary~\ref{cor:epsilon} by induction:
\begin{corollary}\label{cor:gamma-t}
If $\Gamma$ is a $t^{\iinfty}_k$-clasper on $L$, then $\beta^k(L_{\Gamma})=\omega(\Gamma)+\beta^k(L)$, and if $i\neq k$, then $\beta^i(L)=\beta^i(L_{\Gamma})$. $\hfill\square$

\end{corollary}

\begin{figure}
\begin{center}
\includegraphics[width=.4\linewidth]{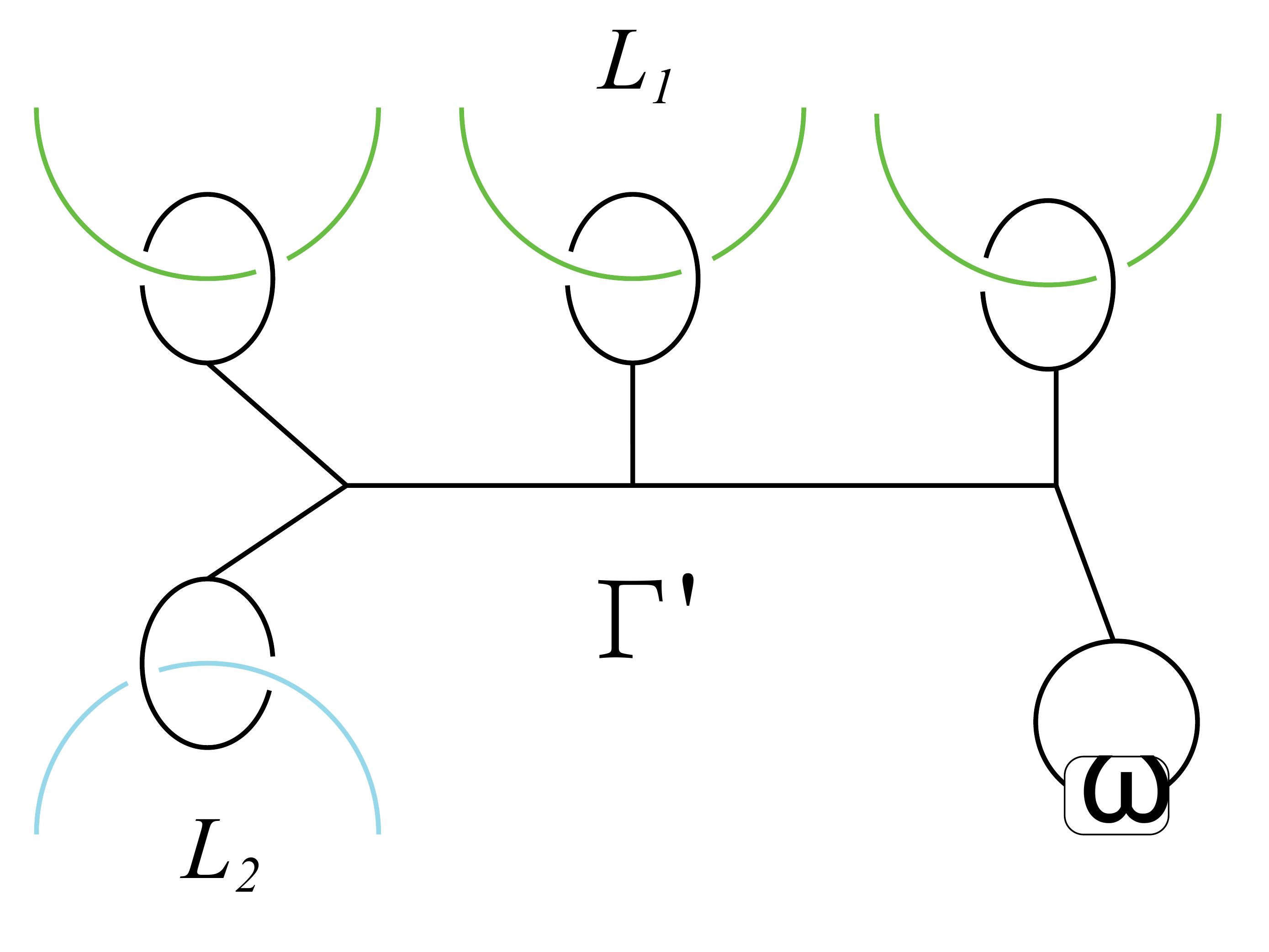}\hspace{2em}
\includegraphics[width=.45\linewidth]{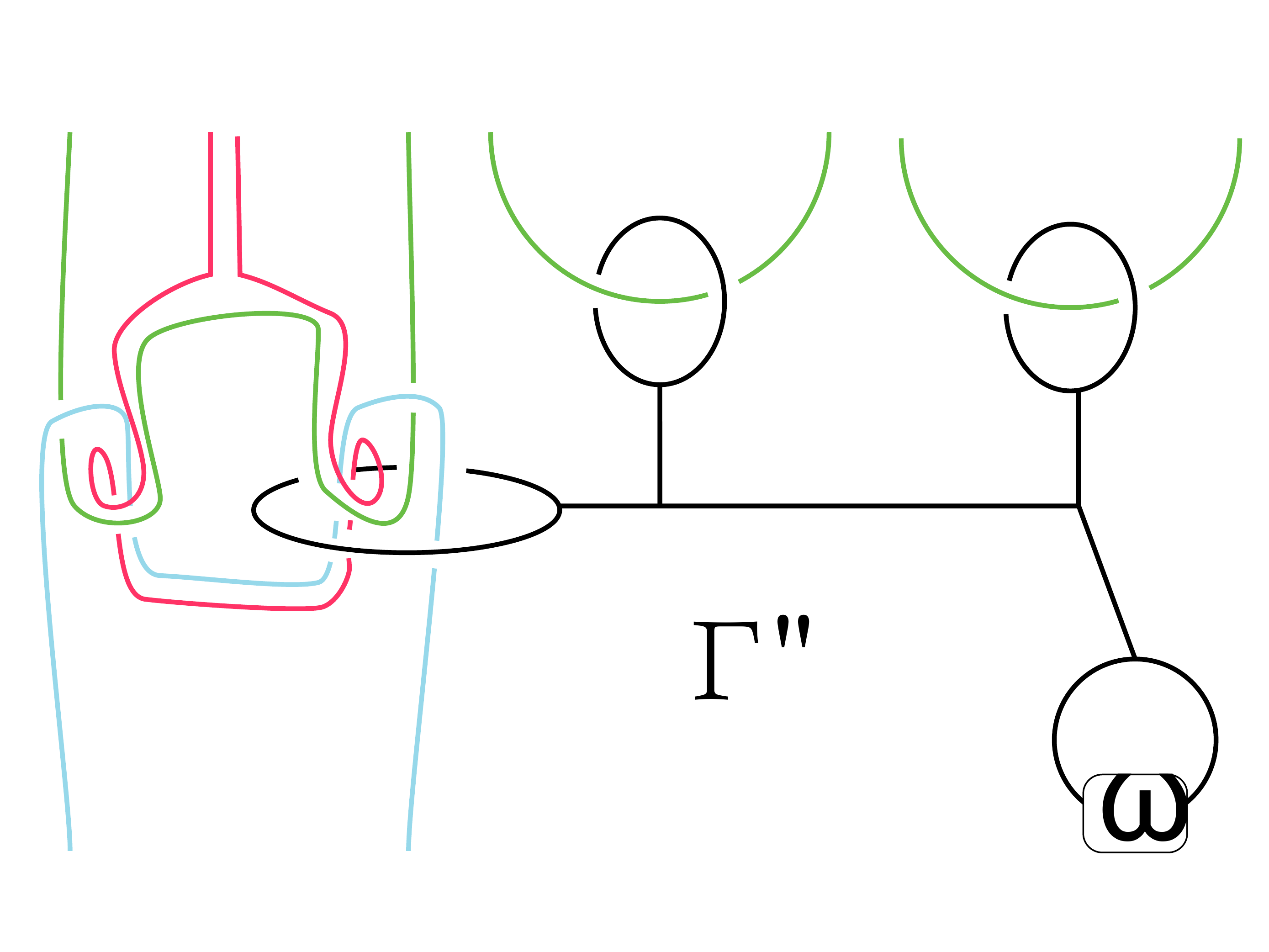}\\
\includegraphics[width=.4\linewidth]{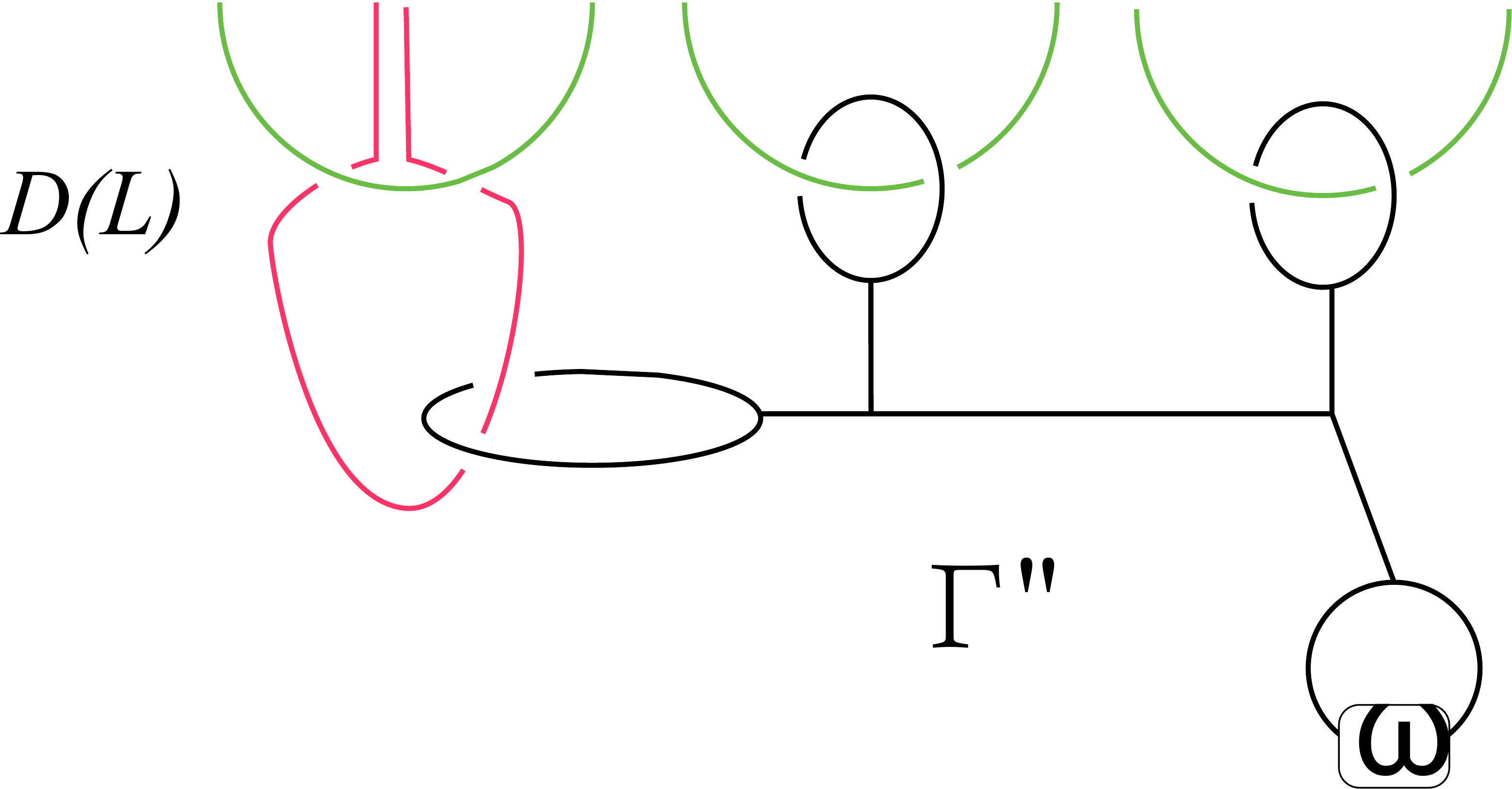}
\end{center}
\caption{
The case $k=3$ from the proof of Proposition~\ref{prop:induct}: 
Upper left: The interior of the $t^{\infty}_3$-clasper $\Gamma'$ on $L$ is disjoint from both Seifert surfaces (which are suppressed from view). Green arcs are from $L_1$ and the blue arc is from $L_2$. 
Upper right: Splitting off a Y-subtree by inserting a Hopf-pair of leaves into an edge of $\Gamma'$ and then doing the Y-surgery on $L$ yields $L_\textrm{Y}$, which is isotopic to $L$, and a further surgery on the clasper $\Gamma''$ would also yield $L_{\Gamma'}$. The red curve shows the intersection between extensions of the Seifert surfaces for $L$ to Seifert surfaces for $L_\textrm{Y}$ (which has been band-summed to the original characteristic curve, as per Remark~\ref{rem:connect}).
Lower: The clasper $\Gamma''$ is a $t^{\infty}_2$-clasper on $D(L)$, with $\omega(\Gamma'')=\omega(\Gamma')$.  Removing $\Gamma''$, the new characteristic curve is isotopic to the old one since it contracts along the band. Hence $D(L_{\Gamma'})=D(L)_{\Gamma''}$.}
\label{fig:main-induction}
\end{figure}

\begin{proposition}\label{prop:twisted-zero-trees}
If $T$ is a twisted tree which has exactly one $2$-label, such that $T\neq t^{\iinfty}_k$ for any $k$, and $\Gamma$ is any $T$-clasper on $L$, then $\beta^i(L)=\beta^i(L_\Gamma)$ for all $i\geq 1$.
\end{proposition}
\begin{proof}
By Lemma~\ref{lem:clasper-off-seifert-surfaces}, we may assume that the interior of $\Gamma$ is disjoint from the surfaces $\Sigma_1$ and $\Sigma_2$. The proof will proceed by covering the three possible cases:

A pair of vertices is said to be \emph{dual} if they are connected by edges to a common trivalent vertex.

\textbf{Case 1:} The $2$-labeled vertex of $T$ is not dual to any univalent vertex. Insert Hopf-pairs of leaves into two edges of $\Gamma$ to split $\Gamma$ into three claspers, one of which is a Y-clasper having a leaf linking $L_2$. (Surgery on all three of these claspers is equivalent to surgery on $\Gamma$.) The top pictures in Figure~\ref{fig:twisted-zero-trees} show 
the result of performing the Y-surgery using the leaf linking $L_2$ as a root. As can be seen in the upper right picture of Figure~\ref{fig:twisted-zero-trees},  
the original Seifert surfaces for $L$ can be extended to Seifert surfaces on $L_{\Gamma}$ without changing the characteristic curve: The indicated (blue) genus one piece is added to $\Sigma_2$ in the complement of the other two claspers. And after performing the other two clasper surgeries, $\Sigma_1$ (green) can be extended to a Seifert surface for the first component of $L_{\Gamma}$ inside a neighborhood of these claspers, so no new intersections are created. Thus $D(L_\Gamma)=D(L)$ and so $\Gamma$-surgery preserves all $\beta^i$.

\textbf{Case 2:} 
The $2$-labeled vertex of $T$ is dual to the $\iinfty$-labeled vertex. 
First note that in this case the $2$-labeled vertex can not also be dual to a $1$-labeled vertex, by the assumption $T\neq t^\iinfty_1$. Split $\Gamma$ into two claspers, one of which is a Y-clasper with one twisted leaf and one leaf linking $L_2$.   
The middle pictures in Figure~\ref{fig:twisted-zero-trees} show 
the result of performing the Y-surgery using the leaf linking $L_2$ as a root. As can be seen in the middle right picture of Figure~\ref{fig:twisted-zero-trees},  
the original Seifert surfaces for $L$ can be extended to Seifert surfaces on $L_{\Gamma}$ without changing the characteristic curve: The indicated (blue) genus one piece with the $\omega(\Gamma)$-twisted band is added to $\Sigma_2$ in the complement of the other clasper. And after performing the other clasper surgery, $\Sigma_1$ (green) can be extended to a Seifert surface for the first component of $L_{\Gamma}$ inside a neighborhood of this clasper, so no new intersections are created. Thus $D(L_\Gamma)=D(L)$ and so $\Gamma$-surgery preserves all $\beta^i$.

\textbf{Case 3:} 
The $2$-labeled vertex of $T$ is dual to a $1$-labeled vertex. 
Note that in this case the $2$-labeled vertex can not also be dual to the $\iinfty$-labeled vertex, by the assumption $T\neq t^\iinfty_1$.
Split $\Gamma$ into two claspers, $\Gamma'$ and a Y-clasper with one leaf linking $L_2$, as in the bottom pictures of Figure~\ref{fig:twisted-zero-trees}.
The bottom right picture in Figure~\ref{fig:twisted-zero-trees} shows 
the result of performing the Y-surgery using the leaf linking $L_2$ as a root.
A Seifert surface $\Sigma'_2$ for the new second component can be constructed by adding the indicated (blue) genus one piece to the original $\Sigma_2$.  Now $\Sigma'_2$ is in the complement of $\Gamma'$, but $\Sigma'_2$ has a ribbon intersection with the original $\Sigma_1$.  
After surgery on $\Gamma'$ (yielding $L_{\Gamma}$), a Seifert surface $\Sigma'_1$ for component $1$ of $L_{\Gamma}$ can be constructed by resolving the ribbon intersection by adding a tube that runs along a subarc of component $2$ of $L_{\Gamma}$, and extending $\Sigma_1$ in a neighborhood of $\Gamma'$.  Now the characteristic curve $\chi'=\Sigma'_1\cap \Sigma'_2$ for $L_{\Gamma}$ has a new loop which runs over one band of $\Sigma'_2$ and links $\Gamma'$. This new loop is connected to the original $\chi$ for $L$ by a band.  
This new $\chi'$ has no additional twists in it, and moreover the derived link $D(L)$ is given by surgery on the clasper $\Gamma'$ on $(L_1,\chi')$ whose tree is of order one less than the order of $T$. Now the assumption $T\neq t^\iinfty_k$ means that we can proceed inductively, since iterating this reduction will eventually lead to Case 1 or Case 2.
\end{proof}

\begin{figure}
\begin{center}
Case 1: \begin{minipage}{8cm}\includegraphics[width=8cm]{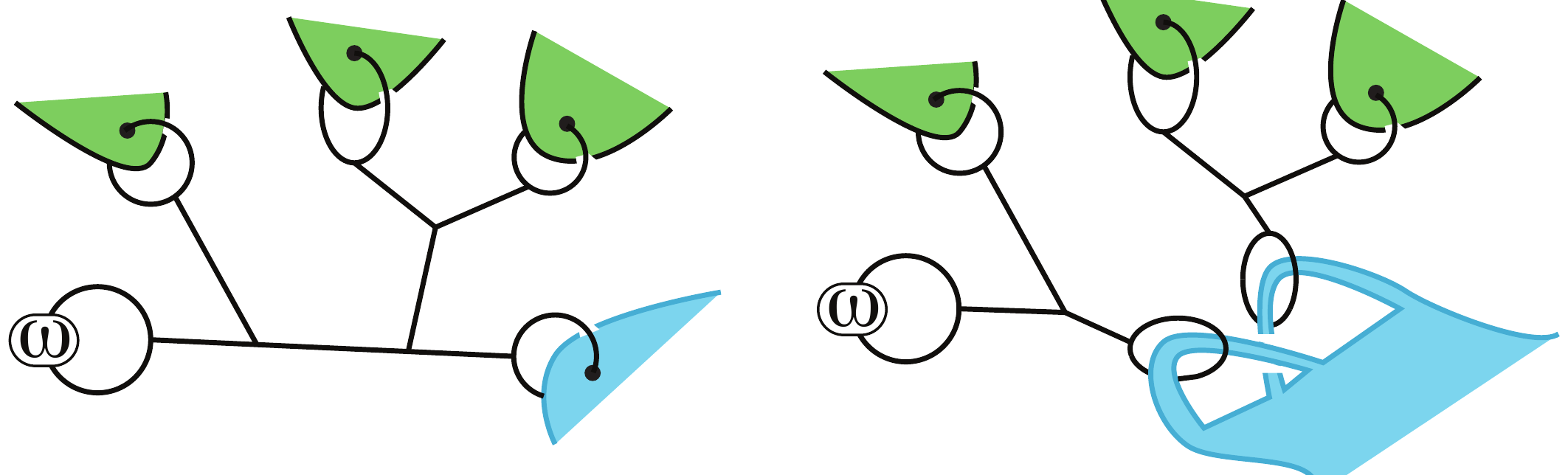}\end{minipage}\\
Case 2:\begin{minipage}{7cm}\includegraphics[width=7cm]{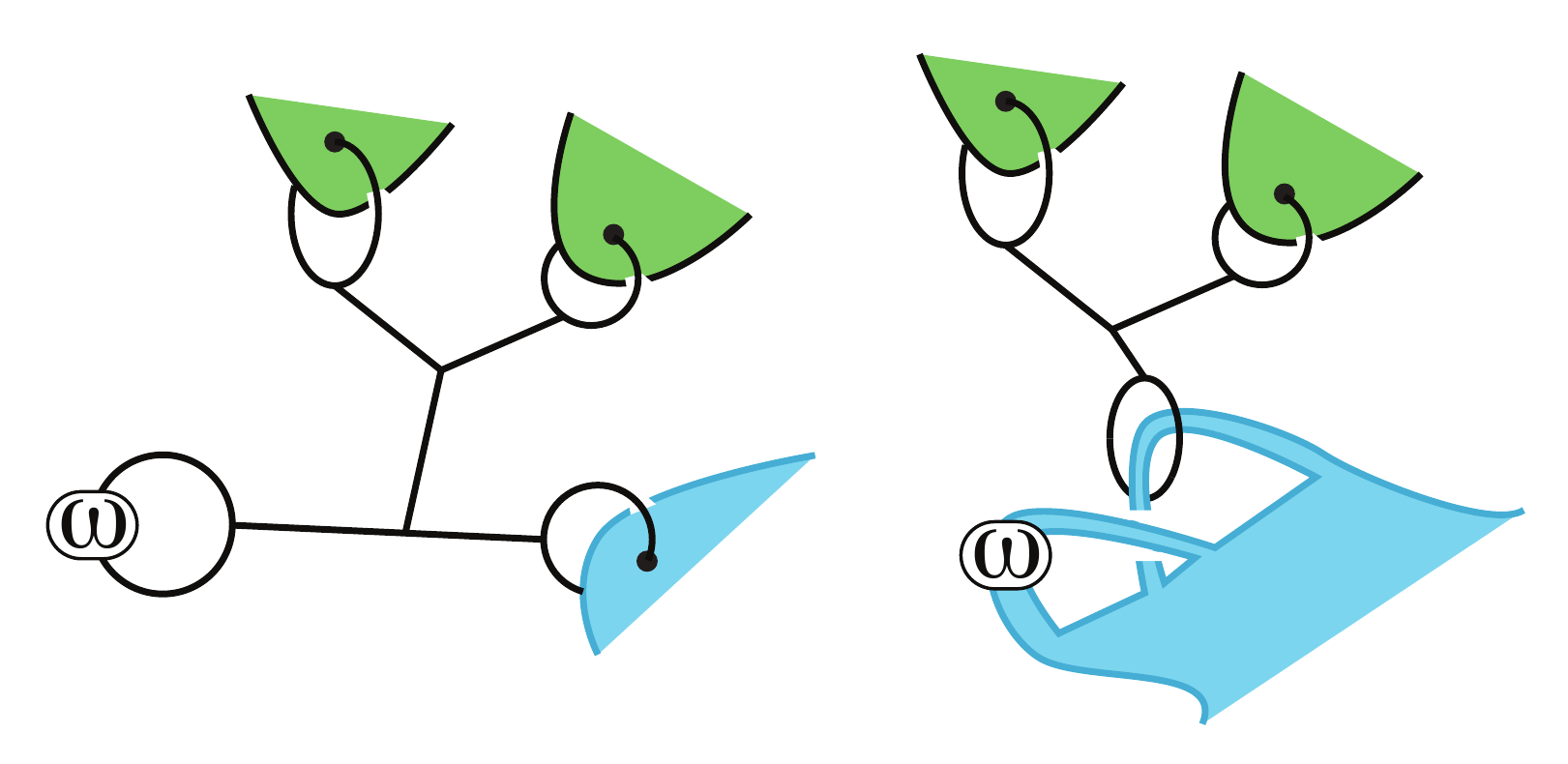}\end{minipage}\\
Case 3:\begin{minipage}{11cm}\includegraphics[width=11cm]{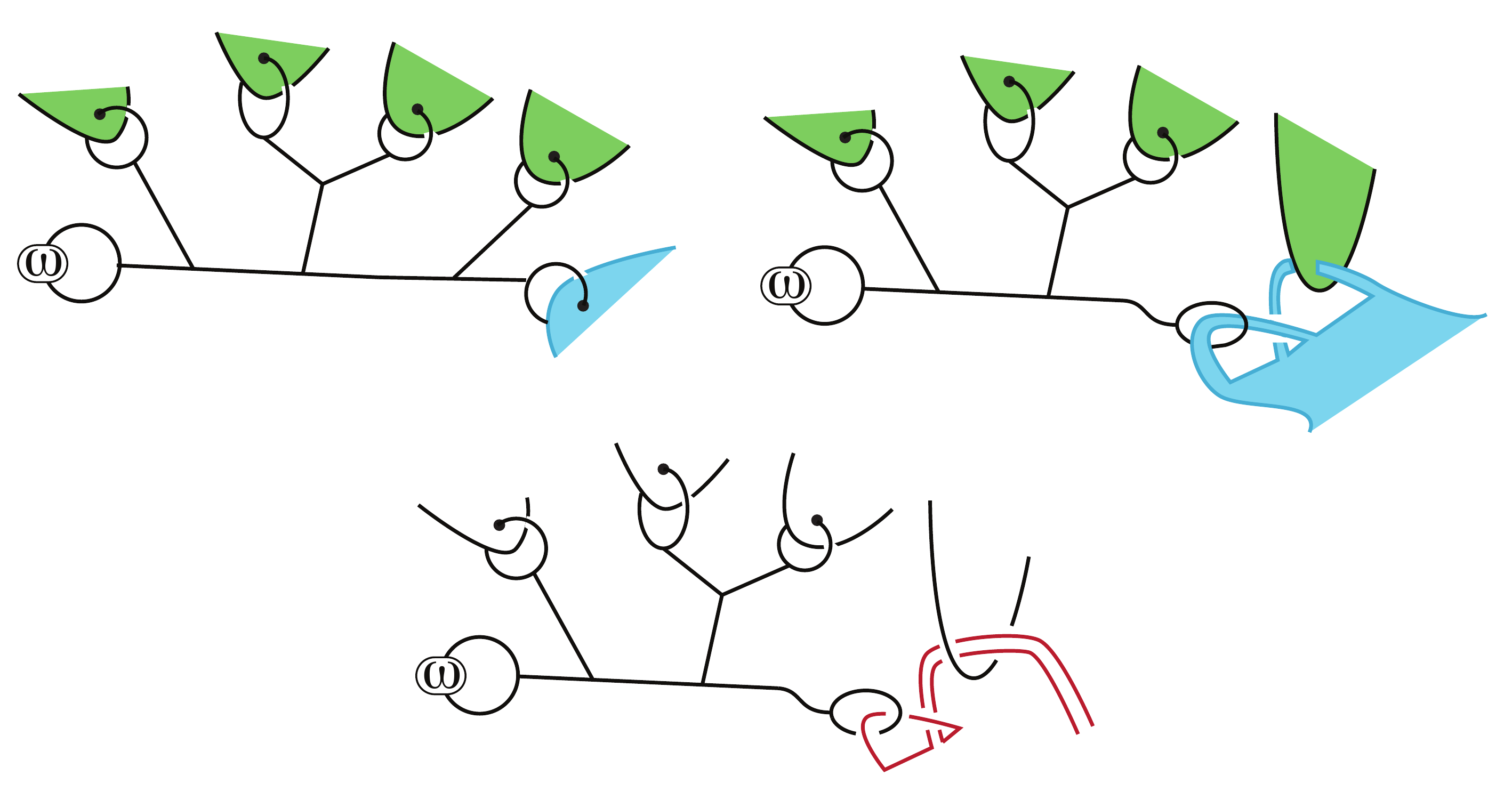}\end{minipage}
\end{center}
\caption{From the proof of Proposition~\ref{prop:twisted-zero-trees}. 
}\label{fig:twisted-zero-trees}
\end{figure}

Recall that $t_n$ is the order $n$ framed tree having two $2$-labels and $n$ $1$-labels as in Figure~\ref{fig:t-tree}.

\begin{proposition}\label{prop:framed-zero-trees}
If $T$ is a framed tree which has two $2$-labels and $n$ $1$-labels, for $n\geq 2$, such that $T\neq t_n$, 
and $\Gamma$ is any $T$-clasper on $L$, then $\beta^i(L)=\beta^i(L_\Gamma)$ for all $i\geq 1$.
\end{proposition}
 
\begin{proof}
Consider such a $T$-clasper on $L$. Split $\Gamma$ by inserting Hopf-pairs of leaves into some edges in such a way that yields a union of five claspers $\Gamma_0,\Gamma_1,\Gamma_2,\Gamma_3,\Gamma_4$ with the following properties: $\Gamma_0$ has 4 leaves. A dual pair of leaves of $\Gamma_0$ each link leaves of $\Gamma_1$ and $\Gamma_2$, while the other dual pair of leaves of $\Gamma_0$ each link $\Gamma_3$ and $\Gamma_4$.  (A pair of leaves is said to be \emph{dual} if they are connected by edges to a common trivalent vertex.) 
Moreover, the other leaves of $\Gamma_1$ and $\Gamma_2$ are meridians to $L_1$, while there is a single leaf of each of $\Gamma_3$ and $\Gamma_4$ which is a meridian to $L_2$. 
Such a splitting of $\Gamma$ is possible because $T\neq t_n$.

Now do surgery on $\Gamma_1\cup \Gamma_2\cup \Gamma_3\cup \Gamma_4$, where we take the roots of  $\Gamma_3$ and $\Gamma_4$ to be the leaves linking $L_2$. This surgery does not change $L$ because each of these claspers has a leaf which bounds a clean cap (although these leaves do link the leaves of $\Gamma_0$). Looking at $\Gamma_0$ on this link, we see that parallel strands of $L_1$ go through one dual pair of leaves of $\Gamma_0$ whereas parallel strands of $L_2$ go through the other dual pair.

Now using the zip construction multiple times to reduce the number of strands through the leaves to $1$ each, surgery on $\Gamma_0$ becomes equivalent to a sequence of surgeries on simple $t$-claspers, with $t= \,^1_1>\!\!\!-\!\!\!\!-\!\!\!\!\!-\!\!\!<^{\,2}_{\,2} $. So it suffices to show that such $t$-surgeries do not affect Cochran invariants.

Consider a $t$-clasper $\Gamma$, with $t= \,^1_1>\!\!\!\!-\!\!\!\!-\!\!\!\!\!-\!\!\!\!<^{\,2}_{\,2} $. We may assume that the caps of $\Gamma$ only meet the Seifert surfaces in four clasp singularities, and 
we may assume that that the interior of $\Gamma$ is disjoint from $\Sigma_2$ by the following argument:
Let $L_\Gamma$ denote the result of $t$-surgery using one of the $2$-labeled leaves as the root. Denote by $H_1$ the homotopy of $L_\Gamma$ which pushes the strand of component $2$ that ran through the non-root leaf across the strands of 
component $2$ that now run along where the non-root leaf used to be. And denote by $H_2$ the homotopy of component $2$ of $L_\Gamma$ induced by pushing $\Gamma$ off of $\Sigma_2$ (so $H_2$ moves strands that run along where $\Gamma$ used to be).
Now, applying
Proposition~\ref{prop:fintype} and noting that
$(L_\Gamma)_{H_1}=L=(L_\Gamma)_{H_1\cup H_2}$ (since the results of $H_1$ and $H_1\cup H_2$ are the same as clasper surgeries on $L$ by claspers that each have a clean leaf) gives:
$$
0=\beta^i(L_\Gamma)-\beta^i((L_\Gamma)_{H_1})-\beta^i((L_\Gamma)_{H_2})+\beta^i((L_\Gamma)_{H_1\cup H_2})=\beta^i(L_\Gamma)-\beta^i((L_\Gamma)_{H_2})
$$
So pushing $\Gamma$ off of $\Sigma_2$ by $H_2$ gets the interior of $\Gamma$ disjoint from $\Sigma_2$ without affecting the 
$\beta^i(L)$.
 
Furthermore, any intersections between $\Sigma_1$  and the interior of $\Gamma$ can all be pushed onto an edge adjacent to a single $2$-labeled leaf: Any intersections between $\Sigma_1$ and the $1$-labeled edges of $\Gamma$ can be pushed into $\Sigma_1$, and the resulting ribbon singularities resolved without affecting the characteristic curve. 

One can now draw an explicit picture of the derived link (Figure~\ref{fig:1122tree}). The new intersections of $\Sigma_1$ and $\Sigma_2$ come from the intersections of $\Sigma_1$ with the $2$-labeled edge of $\Gamma$, and are connected by bands to the rest of the characteristic curve. These new curves are all trivial, implying  
that $D(L_{\Gamma})=D(L)$. Furthermore, the twisting of these additions to the characteristic curve are all $0$, so that $\beta^1(L)=\beta^1(L_\Gamma)$. These last two facts are sufficient to show that $\beta^i(L)=\beta^i(L_\Gamma)$ for all $i\geq 1$.
\end{proof}

\begin{figure}
\begin{center}
\includegraphics[width=\linewidth]{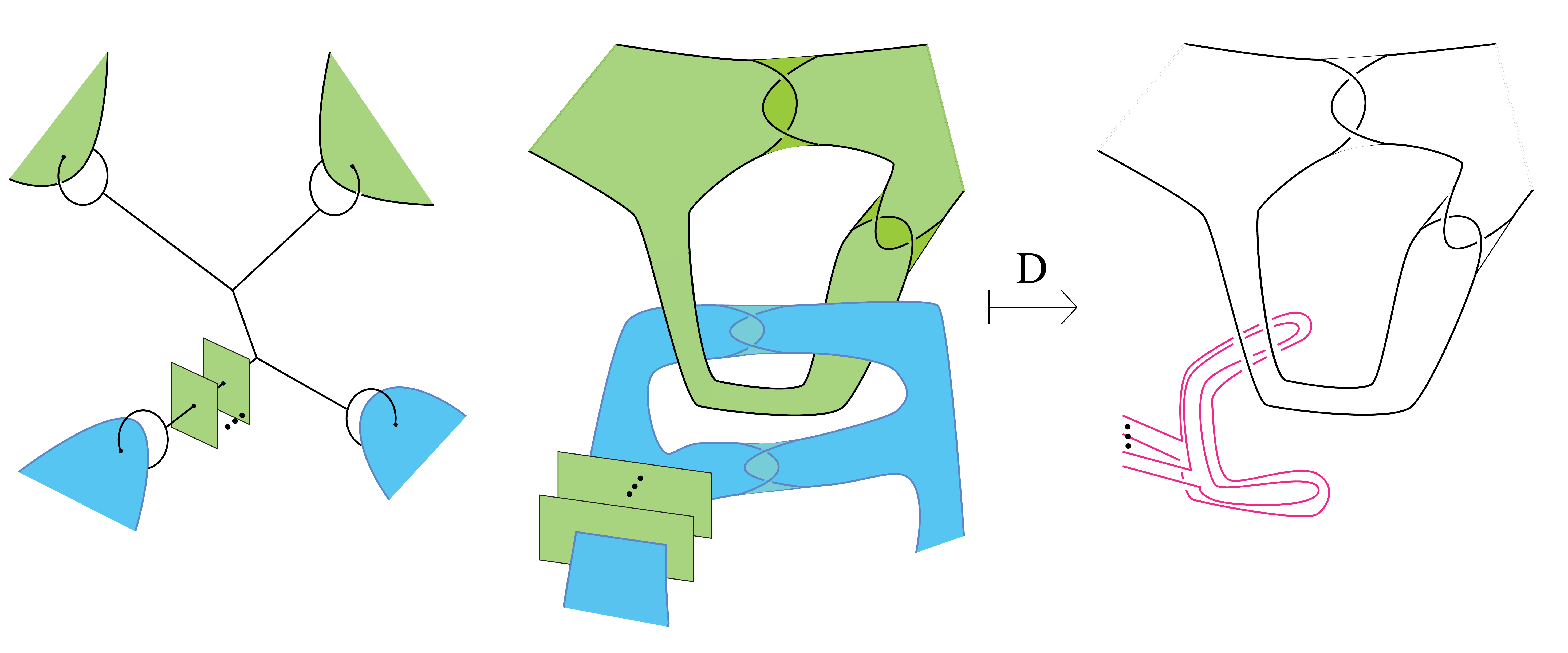}
\end{center}
\caption{
A $t$-surgery on $L$, for $t= \,^1_1>\!\!\!-\!\!\!\!-\!\!\!\!\!-\!\!\!<^{\,2}_{\,2}$, with sheets of $\Sigma_1$ (green) intersecting the clasper (the blue sheets are from $\Sigma_2$). The picture on the right shows $D(L_{\Gamma})=D(L)$, with the characteristic curve in red.
}\label{fig:1122tree}
\end{figure}

\begin{proposition}\label{prop:above-order}
If $T$ is a framed tree of order $> 2k+1$, or a twisted tree of order $> k$, and $\Gamma$ is any $T$-clasper on $L$, then 
$\beta^i(L)=\beta^i(L_\Gamma)$ for all $i\geq 1$.
\end{proposition}
\begin{proof}
First note that for some choice of bands $b$, the band-sum $L\#_b\bar{L}$ is a slice link, which therefore has vanishing $\mu$-invariants in all orders. In particular all $\beta^i$ vanish \cite[Thm.6.10]{Co1}. Now consider $(L\#_b\bar{L})_{\Gamma}=L_{\Gamma}\#_b\bar{L}$. Since $\Gamma$ preserves vanishing of $\mu$-invariants up to order $2k$ \cite{Hab00}, we have that $0=\beta^{i}(L_{\Gamma}\#_b\bar{L})=\beta^i(L_{\Gamma})-\beta^i(L)$ for $i\leq k$.
\end{proof}

\subsection{Proof of Theorem~\ref{thm:main}}\label{subsec:proof-theorem-main}
Finally, we put all the ingredients together to prove Theorem~\ref{thm:main}:
\begin{proof}
First, we show that any $L=(L_1,L_2)$  with trivial linking number and $\Arf(L_1)=0$ bounds a Cochran tower of arbitrarily high order. 

Start by taking any pair of properly immersed disks bounding the components of $L$. Since the linking number is zero, all intersections between the two disks will algebraically cancel, and so can be paired by Whitney disks (which will in general intersect each other and the two immersed disks). By performing local cusp homotopies, it can also be arranged that the disks' self-intersections are all paired by Whitney disks. This yields a twisted Whitney tower $\cW$ of order one. 
If there are no $\beta$-bad trees in $t(\cW)$ then we have a Cochran tower (of infinite order) and we are done. Otherwise, there exist $\beta$-bad trees. 
 
Throughout the following arguments for exchanging $\beta$-bad trees for higher-order trees, we use the ``order-raising'' obstruction theory
described in
\cite{CST}. We refer the reader to that paper for details, including orientation and sign conventions giving rise to coefficients for the trees in $t(\cW)$, as well as descriptions of the various twisting and IHX modifications of $\cW$.
In particular, whenever a pair of intersections points or twisted Whitney disks contribute isomorphic but oppositely-signed trees to $t(\cW)$, then that pair can be eliminated at the cost of only creating new higher-order trees in $t(\cW)$. 
Discussion of signs will be suppressed in the following constructions.

Any $\beta$-bad tree is one of the following four types:
 
{\bf The trees $\tree{1}{1}{1}$ and $\tree{1}{1}{\iinfty}$:}   
The trees $\tree{1}{1}{\iinfty}$ can all be exchanged for more of the trees $\tree{1}{1}{1}$ by the \emph{boundary-twisting} operation \cite[p.1455]{CST}. And since $\Arf(L_1)=0$, the trees $\tree{1}{1}{1}$ in $t(\cW)$ have to appear an even number of times \cite{FK,Ma}. Such trees represent $2$-torsion (by antisymmetry), so all these trees
can be eliminated from $t(\cW)$ at the cost of creating only higher-order trees.

{\bf The trees $t_{2k-1}$:} 
Applying the \emph{boundary twist} operation to a clean framed order $k$ Whitney disk $W_{(\cdots(((2,1),1),1),\cdots,1)}$ created by finger moves creates an intersection corresponding to the tree $t_{2k-1}$ at the cost of also creating a twisted Whitney disk whose associated tree is $t^\iinfty_k$. So, for any $k$, all $t_{2k-1}$-trees can be eliminated from $t(\cW)$, at the cost of only creating new $t^\iinfty_k$-trees and higher-order trees.

  {\bf The trees $t_{2k}$:} 
Applying the \emph{interior twist} operation
\cite[p.1456]{CST} to a clean order $k$ framed Whitney disk $W_{(\cdots(((2,1),1),1),\cdots,1)}$ created by finger moves creates an intersection corresponding to the tree $t_{2k}$ at the cost of only creating a $2$-twisted Whitney disk whose associated tree is $t^\iinfty_k$.  So, for any $k$, all $t_{2k}$-trees can be eliminated from $t(\cW)$, at the cost of only creating new $t^\iinfty_k$-trees and higher-order trees.

 {\bf Trees which are zero modulo IHX}:  
As described in \cite{CST}, using the (twisted) IHX construction, any tree which represents zero modulo IHX relations can be eliminated from $t(\cW)$ at the cost of creating only higher-order trees. 

Thus we can always remove $\beta$-bad trees at the cost of creating only higher-order trees, which is sufficient to establish the desired result inductively.

Secondly, we show that the $\beta^i(L)$ are the signed count of $t_i^{\iinfty}$ trees.
Decompose the order $2k$ Cochran tower into concordances and simple (twisted) clasper surgeries from the unlink to $L$, by Theorem~\ref{thm:clasper-concordance}. As shown by Cochran, concordances leave $\beta^i$ invariant. If a tree is not $\beta$-bad, then a corresponding clasper surgery can only change $\beta^i$ if it is a $t_i^{\iinfty}$-surgery, by Corollary~\ref{cor:three-twos}, Proposition~\ref{prop:framed-zero-trees} and Proposition~\ref{prop:twisted-zero-trees}. And by Corollary~\ref{cor:gamma-t}, the effect of a $t^{\iinfty}_i$-surgery is to change $\beta^i(L)$ exactly by $\omega(\Gamma)$.  Finally, clasper surgeries whose trees have order $>2k$ and twisted clasper surgeries whose trees have order $>k$ don't affect $\beta^i$ for $i\leq k$, by Proposition~\ref{prop:above-order}.
\end{proof}

\begin{remark}\label{rem:show-badness}
The construction of a Cochran tower of arbitrary height in the proof of Theorem~\ref{thm:main} follows from the fact that the Whitney tower obstruction theory allows for $\beta$-bad trees to be exchanged for $t_i^{\iinfty}$ trees and higher-order trees. Although we have given examples of how certain $\beta$-bad trees can create indeterminacies in the computation of Cochran invariants as the count of twistings on Whitney disks associated to $t_i^{\iinfty}$ trees (Examples \ref{ex:infmany}, \ref{ex:no-2-labels-are-bad} 
and \ref{ex:t2-bad}), it is possible that some of the trees we have defined to be $\beta$-bad might not create indeterminacies in this computation. 
Eliminating some trees from the $\beta$-bad list would \emph{a priori} enlarge the set of links that bound Cochran towers of infinite order, a particularly nice case where all the Cochran invariants can be computed from a single Cochran tower. 
\end{remark}

\section*{Acknowledgments}
The authors thank Max-Planck-Institut f\"ur Mathematik for its generous hospitality and support. This paper originated there during discussions in Summer 2015 when the first two authors were visiting the third. The second author is also supported by a Simons Foundation \emph{Collaboration Grant for Mathematicians}.
 
The authors would also like to thank Tim Cochran for life-long inspiration.

\end{document}